\documentclass{amsart}
\usepackage[english]{babel}     
\usepackage[utf8]{inputenc}
\usepackage{amsmath}      
\usepackage{amssymb}     
\usepackage{mathrsfs} 
\usepackage{stmaryrd}
\usepackage{hyperref}
\usepackage{tikz}
\usepackage{tikz-cd}
\usetikzlibrary{patterns}
\usepackage{lscape}
\usepackage{todonotes}
\usepackage[T1]{fontenc}
\newtheorem{theorem}{Theorem}[section]
\newtheorem{lemma}[theorem]{Lemma}
\theoremstyle{definition}
\newtheorem{definition}[theorem]{Definition}
\newtheorem{example}[theorem]{Example}
\newtheorem{corollary}[theorem]{Corollary}

\newtheorem{proposition}[theorem]{Proposition}
\newtheorem{conjecture}[theorem]{Conjecture}
\DeclareSymbolFont{Shuffle}{U}{shuffle}{m}{n}
\DeclareFontFamily{U}{shuffle}{}
\DeclareFontShape{U}{shuffle}{m}{n}{%
  <-8>shuffle7%
  <8->shuffle10%
}{}
\DeclareMathSymbol\shuffle{\mathbin}{Shuffle}{"001}
\DeclareMathSymbol\cshuffle{\mathbin}{Shuffle}{"002}
\theoremstyle{remark}
\newtheorem{remark}[theorem]{Remark}
\numberwithin{equation}{section}

\newcommand{\att}{\big \vert}

\newcommand{\pr}			{{\operatorname{\mathsf{pr}}}}

\newcommand{\born}		{{\operatorname{\mathrm{born}}}}
\newcommand{\barr}		{{\operatorname{\mathrm{bar}}}}

\newcommand{\dR}		{{\operatorname{\mathrm{dR}}}}
\newcommand{\sing}		{{\operatorname{\mathrm{sing}}}}
\newcommand{\fl}		{{\operatorname{\mathrm{flat}}}}

\newcommand{\hotimes}	{\mathbin{\widehat{\otimes}}}
\newcommand{\action} { \rotatebox[origin=c]{-90}{$\circlearrowright$}}
\DeclareMathOperator{\im}{Im}
\DeclareMathOperator{\gl}{\mathfrak{gl}}
\DeclareMathOperator{\shl}{\mathfrak{sl}}
\DeclareMathOperator{\shp}{\mathfrak{sp}}
\DeclareMathOperator{\so}{\mathfrak{so}}

\DeclareMathOperator{\GL}{GL}

\DeclareMathOperator{\supp}{supp}
\DeclareMathOperator{\Ad}{Ad}

\DeclareMathOperator{\id}{id}

\DeclareMathOperator{\tr}{tr}

\DeclareMathOperator{\End}{End}

\DeclareMathOperator{\coker}{coker}
\DeclareFontFamily{U}{MnSymbolC}{}
\DeclareSymbolFont{MnSyC}{U}{MnSymbolC}{m}{n}
\DeclareFontShape{U}{MnSymbolC}{m}{n}{
    <-6>  MnSymbolC5
   <6-7>  MnSymbolC6
   <7-8>  MnSymbolC7
   <8-9>  MnSymbolC8
   <9-10> MnSymbolC9
  <10-12> MnSymbolC10
  <12->   MnSymbolC12}{}
\DeclareMathSymbol{\intprod}{\mathbin}{MnSyC}{'270}

\DeclareFontFamily{U}{MnSymbolC}{}
\DeclareSymbolFont{MnSyC}{U}{MnSymbolC}{m}{n}
\DeclareFontShape{U}{MnSymbolC}{m}{n}{
    <-6>  MnSymbolC5
   <6-7>  MnSymbolC6
   <7-8>  MnSymbolC7
   <8-9>  MnSymbolC8
   <9-10> MnSymbolC9
  <10-12> MnSymbolC10
  <12->   MnSymbolC12}{}
\DeclareMathSymbol{\intprod}{\mathbin}{MnSyC}{'270}
\begin{document}
\title[Bornological Loday-Quillen-Tsygan theorems]{Continuous cohomology of gauge algebras and bornological Loday-Quillen-Tsygan theorems}
\author{Lukas Miaskiwskyi}
\curraddr{Utrecht, The Netherlands
}
\email{lukas.mias@gmail.com}
\date{17-6-2022}
\begin{abstract}
We investigate the well-known Loday-Quillen-Tsygan theorem, which calculates the Lie algebra homology of the general linear algebra~$\gl(A)$ for an associative algebra~$A$ in terms of cyclic homology, and extend the proof to bornological Lie algebra homology of Fréchet and LF-algebras. For Fréchet spaces, this equals continuous Lie algebra homology.
 To this end we prepare several statements about homological algebra of topological vector spaces, and discuss when the differential of the bornological Hochschild and cyclic complex are topological homomorphisms in the setting of Fréchet algebras. 
We apply the results to the algebras of smooth functions on a smooth manifold and compactly supported smooth functions on Euclidean space, and construct from a local-to-global principle a Gelfand-Fuks-like spectral sequence which calculates the stable part of bornological Lie algebra homology of nontrivial gauge algebras. This complements results by Maier, Janssens and Wockel.
\end{abstract}
\maketitle
\tableofcontents
\section{Introduction}
In recent years, interest has grown to understand continuous Lie algebra co\-ho\-mo\-lo\-gy of certain infinite-dimensional Lie algebras, specifically low-degree cohomology, which contains information about continuous central extensions of the given Lie algebra and thus its projective representations. 
In degree~$\leq 2$, the cohomological equations can reasonably well be handled explicitly, so many interesting examples from the realm of dif\-fe\-ren\-tial/sym\-plec\-tic geometry are now well-understood, see for example \cite{janssens2016universal}, \cite{neeb2008second}, \cite{neeb2009central}.
\\
Of specific interest in physics is the infinite-dimensional symmetry group of gauge transformations, and its corresponding infinitesimal symmetries, modelled as the sections of a Lie algebra bundle~$\mathcal{K} \to M$. This is what we call a \emph{gauge algebra}. 
When the fibres of this bundle are semisimple finite-dimensional Lie algebras, the second degree cohomology is fully understood, see \cite{maier2001central} for the globally trivial case and \cite{janssens2013universal} for the general case of nontrivial gauge algebras. 
However, the methods within these papers are very specifically suited to degree 2, which raises the question of how one might calculate higher degree cohomology.
\\
Independently, roughly 40 years ago, Loday, Quillen and Tsygan fully described the algebraic Lie algebra homology of~$\gl(A) = \varinjlim \left(\gl_n(\mathbb{K}) \otimes A\right)$ for arbitrary unital algebras~$A$ and fields~$\mathbb{K}$ with~$\mathbb{Q} \subset \mathbb{K}$ in terms of cyclic homology~$H_\bullet^\lambda(A)$, see \cite{loday1984cyclic}, \cite{tsygan1983homology}.
 Their proof lays the groundwork for results about the homology of many so-called \emph{current algebras}, Lie algebras of the shape~$\mathfrak{g} \otimes A$, where~$\mathfrak{g}$ is another Lie algebra and~$A$ is an associative algebra. 
 In particular, when~$\mathfrak{g}$ equals any of the classical simple Lie algebras, their method allows one to extract quite a lot of information.
\\
Now, if~$\mathfrak{g}$ is finite-dimensional, the current algebra~$\mathfrak{g} \otimes C^\infty(M) \cong C^\infty(M,\mathfrak{g})$ represents exactly the gauge algebra of a globally trivial Lie algebra bundle with fibres equal to~$\mathfrak{g}$, establishing a connection between the work of Loday, Quillen and Tsygan, and the study of gauge algebras. 
However, since one is in general not only interested in the algebraic Lie algebra cohomology of gauge algebras, but their continuous counterpart, one may ask the question if the proof of the Loday-Quillen-Tsygan (LQT) result holds when the involved homology theories are modified to take topological data into account. 
This would provide a unified way to calculate continuous (co-)homology of locally trivial gauge algebras with many different fibre Lie algebras, providing information in more than just low degree. 
\\
The goal of this paper is to explore this question and answer it in the affirmative for \emph{bornological} Lie algebra homology. On the~$\gl(A) = \varinjlim \gl_n(A)$, this is essentially Lie algebra homology defined in terms of Grothendieck's completed inductive topological tensor product, rather than the more standard projective tensor product. 
Due to the fact that this tensor product is compatible with the ubiquitous direct limit arguments in the proof of the LQT theorem, this appears to be the most natural framing for a topological LQT theorem.
\\
Note that in \cite{feigin1988cohomology}, a result for the continuous cohomology of~$\gl(C^\infty(M))$ for closed manifolds~$M$ is stated, but lacking a full proof. They state that the jointly continuous cohomology is freely generated by the continuous cyclic cohomology of the algebra; complementary to that, our result is that the bornological homology is equal to the topological completion of this freely generated space. While we cannot disprove their claim outright due to the difference in continuity notions, it seems likely that also the statement in joint continuity should include such a completion, as the generators contain infinite-dimensional components.
\\
We begin by laying out the foundation of this study: we recall in Section~\ref{SectionTopologicalVectorSpaces} and~\ref{SectionContinuousHochschildAndCyclicHomology} the definition and important properties of topological vector spaces, e.g. Fréchet and LF-spaces, their associated topological tensor products, and bornological homology theories for associative algebras. In particular, we prepare statements about the Fréchet algebra~$A = C^\infty(M)$ and the LF-algebra~$A = C^\infty_c(\mathbb{R}^n)$ and prove that the Hochschild and cyclic differentials are topological homomorphisms in these cases.
\\
In Section~\ref{SectionLodayQuillenFrechet}, the algebraic LQT-Theorem is extended to bornological Lie algebra homology and (homologically) unital Fréchet algebras in Theorems~\ref{thm:NuclearFrechetAlgebrasFulfilLQT} and~\ref{TheoremNonUnitalLQT}. In essence, this requires tracking through the algebraic proofs and making sure that all algebraic isomorphisms lift to topological isomorphisms in the respective topologies and on the completions of the tensor products. As an application, in Corollary~\ref{cor:LQTForSmoothFunctions} we fully state the bornological Lie algebra homology of~$\gl(C^\infty(M)) = \varinjlim \gl_n(C^\infty(M))$ and~$\gl(C^\infty_c(\mathbb{R}^n)) = \varinjlim \gl_n(C^\infty_c(\mathbb{R}^n))$, both spaces equipped with their respective direct limit topologies.
\\
Lastly, in Section~\ref{SectionNontrivialGaugeAlgs} we globalize our results to approximate the bornological Lie algebra homology of gauge algebras~$\Gamma(\Ad P \to M)$ for principal bundles~$P \to M$. We restrict the calculations to when the fibre Lie algebra is~$\gl_n(\mathbb{K})$, but the general method is easily transferrable to other classical, simple Lie algebras.  We construct in Theorem~\ref{TheoremSpectralSequenceForGammaK} a spectral sequence which calculates this homology in stable degree. This is parallel to a well-known construction by Gelfand, Fuks, Bott and Segal, who have constructed such spectral sequences  for continuous Lie algebra cohomology of vector fields on a smooth manifold \cite{bott1977cohomology} \cite{gelfand1969cohomologies}. 
Unfortunately, the entries of the second page of the spectral sequence can be specified only in terms of a certain \v{C}ech homology of product cosheaves, which we are unable to calculate and can only conjecture. This is due to the lack of a Künneth theorem in the cosheaf-theoretical setting. Assuming this conjecture, however, this spectral sequence yields a unified approach to compute low-dimensional bornological cohomology of a large class of gauge algebras. An example of such results is given in Corollary~\ref{cor:StableCohomologyInSpecialCases}.
 
A related approach is given in \cite{gwilliams2018higher}, which does not consider continuous, but \emph{local} Loday-Quillen-Tsygan Theorems, in the language of factorization algebras.
\noindent {\it Acknowledgement.}  The author was supported by the NWO Vidi Grant 639.032.734 ``Cohomology and representation theory of infinite-dimensional Lie groups''.   
\section{Topological vector spaces, bornologies, and tensor products}
\label{SectionTopologicalVectorSpaces}
\subsection{Preliminaries and definitions}
We want to begin by collecting some definitions and results regarding topological vector spaces and their tensor products. For a more detailled discussion, we direct the reader to \cite{treves1967topological}, \cite{schaefer1971locally}, \cite{meise1997einfuhrung}. 
Fix, once and for all, a topological field~$\mathbb{K}$ containing the rational numbers. All vector spaces and algebras in the following will be over~$\mathbb{K}$ unless specified otherwise.
\begin{definition}
Let~$V$ be a vector space.
\begin{itemize}
\item[i)] We call~$V$ a \emph{topological vector space (TVS)}  if it is equipped with equipped with a topology in which addition~$V \times V \to V$ and scalar multiplication~$V \times V \to V$ are continuous with respect to the according (product) topologies.
\item[ii)] We call~$V$ \emph{locally convex vector space (LCTVS)} if it is a TVS which is Hausdorff and in which every point has a neighbourhood basis consisting of convex sets.
\item[iii)] We call~$V$ a \emph{Fréchet space} if it is a metrizable and complete LCTVS.
\item[iv)] We call~$V$ an \emph{LF-space} if it is the inductive limit of a countable direct system of Fréchet spaces~$(V_n)_{n \in \mathbb{N}}$, equipped with the inductive limit topology.\\ 
Additionally, it is a \emph{strict} LF-space if the maps~$V_n \to V_m$ in the direct system (for~$m \geq n$) are topological embeddings. \\
In both cases, we write this as~$V = \varinjlim V_n$.
\item[v)] We call~$V$ a \emph{topological algebra} if it is a TVS with a continuous multiplication~$\mu : V \times V \to V$ that makes~$V$ into an associative~$\mathbb{K}$-algebra. 
\item[vi)] We call~$V$ a \emph{topological Lie algebra} if it is a TVS with a continuous map~$[\cdot,\cdot] : V \times V \to V$ that makes~$V$ into a~$\mathbb{K}$-Lie algebra. 
\end{itemize} 
\end{definition}
%
%
%
\begin{remark}
We will mix and match with the above terms when useful; for example, we will work with topological algebras where the underlying TVS is Fréchet or an LF-space. Then we will simply call it a \emph{Fréchet algebra} or an \emph{LF-algebra}, respectively.
\end{remark}
\begin{example}
Let~$M$ be a smooth manifold. Our main Fréchet space of interest will be the space of smooth functions on~$M$, denoted~$C^\infty(M)$. It is well-known to admit a topology that makes it into a Fréchet space. 
\\
This topology is sequential, and a sequence~$(f_n \in C^\infty(M))_{n \in \mathbb{N}}$ converges to~$f \in C^\infty(M)$ if and only if all (locally defined) derivatives of the~$f_n$ uniformly converge to the derivatives of~$f$ on all compact sets~$K$ contained within charts. The standard pointwise multiplication is easily shown to be continuous with respect to this topology, making it into a Fréchet algebra.
\\
Straightforwardly, these considerations can be extended to give the space of sections~$\Gamma(E)$ of a finite-dimensional vector bundle~$E \to M$ a Fréchet space struture, but, in general, with no available algebra structure.
\end{example}
\begin{example}
An important~$LF$-space for us will be~$\gl(\mathbb{K}) := \varinjlim \gl_n(\mathbb{K})$. Since all~$\gl_n(\mathbb{K})$ are finite-dimensional, they admit canonical Fréchet space structures, and they make up a direct system via the inclusions~$\gl_n(\mathbb{K}) \to \gl_m(\mathbb{K})$ for~$n \leq m$. 
The Lie brackets are compatible with these inclusion maps, and one can show that the arising Lie bracket on~$\gl(\mathbb{K})$ is continuous with respect to the LF-topology.
\end{example}
\subsection{Topological tensor products}
In contrast to the finite-dimensional case, tensor products become very delicate in infinite dimensions. Certainly one always has the algebraic tensor product~$\otimes = \otimes_{\mathbb{K}}$, but there are multiple non-equivalent ways to equip this with a topology. For our purposes we will recall two such notions.
\begin{definition} Let~$V, W$ be two LCTVS and consider the canonical map
\begin{align*}
\phi : V \times W \to V \otimes W.
\end{align*}
\begin{itemize}
\item[i)] The \emph{projective tensor product}~$V \otimes_\pi W$ denotes the vector space~$V \otimes W$ equipped with the strongest locally convex topology such that~$\phi$ is continuous.
\item[ii)] The \emph{inductive tensor product}~$V \otimes_\iota W$ denotes the vector space~$V \otimes W$ equipped with the strongest locally convex topology such that~$\phi$ is \emph{separately continuous}, meaning~$\phi(v,\cdot)~$ and~$\phi(\cdot,w)$ are continuous for all~$v \in V, w \in W$.
\item[iii)] The \emph{bornological tensor product}~$V \otimes_\beta W$ denotes the vector space~$V \otimes W$ equipped with the strongest locally convex topology such that~$\phi$ is \emph{bounded}, meaning if~$B \subset V \times W$ is bounded, then~$\phi(B) \subset V \otimes_\beta W$ is, too.
\end{itemize}
\end{definition}
\begin{remark}
The tensor products here should not be confused with the \emph{injective} tensor product, generally denoted by~$\otimes_\epsilon$.
\end{remark}
See \cite[Chapter 3.6]{schaefer1971locally}, \cite[Chapter 1]{grothendieck1955produits}, \cite{kriegl1997convenient} for proofs of the existence of these tensor product topologies and additional details.
In our setting, it suffices to work only with~$\otimes_\beta$, and we denote by~$V \hotimes W := \overline{V \otimes_\beta W}$ the completion of the bornological tensor product. We set no notation for the completion with respect to other tensor products.
%
%
%
Topological tensor products appear to have somewhat of a bad reputation, and in complete generality, they may well deserve it. However, in our setting, the categorical properties of the tensor products are fairly pleasant, especially the bornological tensor product:
\begin{proposition} 
\label{prop:PropertiesOfTopTensorProds}
Let~$V,W,U$ be LCTVS.
\begin{itemize}
\item[i)]\cite[Chapter 34.2, Chapter 43]{treves1967topological} \cite[Chapter 5]{kriegl1997convenient} If~$V$ and~$W$ are Fréchet, then all separately continuous bilinear maps~$V \times W \to U$ are continuous,
\begin{align*}
V \otimes_\pi W \cong V \otimes_\iota W \cong V \otimes_\beta W,
\end{align*}
and~$V \hotimes W$ is Fréchet.
\item[ii)] \cite[Appendix A.1.4]{meyer1999analytic} If~$V$ and~$W$ are nuclear strict LF-spaces 
\begin{align*}
V = \varinjlim V_i, \quad W = \varinjlim W_j,
\end{align*}
then~$V \hotimes W$ is a nuclear strict LF-space, and
\begin{align*}
V \otimes_\iota W \cong V \otimes_\beta W, \quad
V \hotimes W = \varinjlim (V_i \hotimes W_j).
\end{align*}
\item[iii)] \cite[Prop 2.25]{meyer1999analytic} We canonically have
\begin{align*}
(U \hotimes V) \hotimes W \cong U \hotimes (V \hotimes W), \quad
U \hotimes V \cong V \hotimes U.
\end{align*}
\end{itemize}
\end{proposition}
\begin{remark}
Note that commutativity follows quite easily using the topological isomorphism~$X \times Y \cong Y \times X$.
However, the associativity is not generally quite so obvious: For general topological tensor products, the natural vector space isomorphisms between~$X \otimes (Y \otimes Z)$ and~$(X \otimes Y) \otimes Z$ might not necessarily be continuous, see \cite{gloeckner2004tensor}.
\end{remark}
Lastly, we have an exactness property of~$\hotimes$. The following is a consequence of \cite[Thm A1.6]{eschmeier1996spectral}:
\begin{proposition}
\label{prop:ProjTensorProductExact}
Let~$U,V,W$ and~$H$ be nuclear Fréchet spaces, and
\begin{align*}
0 \to U \stackrel{f}{\to} V \stackrel{g}{\to} W \to 0
\end{align*}
an exact sequence, in the sense that~$f$ and~$g$ are continuous and linear,~$f$ is injective,~$g$ is surjective, and~$\mathrm{im} f  = \ker g$.
Then
\begin{align*}
0 \to U \hotimes H \stackrel{f \hotimes \id}{\to} V \hotimes H \stackrel{g \hotimes \id}{\to} W \hotimes H \to 0
\end{align*}
is exact in the same sense.
\end{proposition}
\section{Bornological Hochschild and cyclic homology}
\label{SectionContinuousHochschildAndCyclicHomology}
\subsection{Preliminaries and definitions}
In this section, we recall how to modify the algebraic notions of Hochschild/cyclic homology to include topological information. For a succinct presentation of the algebraic picture of Hochschild and cyclic homology and certain topological modifications, we cite \cite{loday1998cyclic}, \cite{khalkhali2009basic}, \cite{connes1985noncommutative}. A related discussion of the topological modifications also takes place in \cite{brodzki2001excision}.
 We lay no claim to originality within this section, with the exception of investigating the property of differentials to be topological morphisms; we will explain this below.
Fix a field~$\mathbb{K}$. All algebras in the following will be~$\mathbb{K}$-algebras.
\begin{definition}
\label{DefinitionHochschildComplex}
Let~$A$ be a topological algebra. The \emph{bornological Hochschild complex} of~$A$ with coefficients in itself is given by
\begin{align*}
HC_\bullet^\born(A) := 
\bigoplus_{k \geq 0} HC_k(A),
\quad
HC_k(A) := A^{\hotimes^{k+1}}.
\end{align*}
where the differential is induced by the \emph{Hochschild differential}, so
\begin{align*}
b : HC_k^\born(A) &\to HC_{k-1}^\born(A), 
\\
b(a_0 \otimes \dots \otimes a_n) :=&
\sum_{i=0}^{n-1} (-1)^{i} a_0 \otimes \dots \otimes a_i a_{i+1} \otimes \dots \otimes a_n
\\
&+ (-1)^{n+1} a_n a_0 \otimes a_1 \otimes \dots \otimes a_{n-1}.
\end{align*}
The homology of this complex is called the \emph{bornological Hochschild homology} of~$A$ and denoted~$HH^\born_\bullet(A)$.
\end{definition}
\begin{remark}
A reader who is less familiar with bornology may instead be interested in replacing~$\otimes_\beta$ with, say, the completion of~$\otimes_\pi$ or~$\otimes_\iota$ rather than~$\hotimes = \overline{\otimes_\beta}$, to get something which could reasonably be called \emph{(jointly) continuous} or \emph{separately} continuous homology. Due to Proposition~\ref{prop:PropertiesOfTopTensorProds}, this is possible whenever~$A$ is, respectively, Fréchet or a strict LF-space. These cases will be studied further in Section~\ref{SectionLodayQuillenFrechet}.
\end{remark}
\begin{remark}
Note that for nonunital algebras~$A$, the above definition is the born\-o\-lo\-gi\-cal version of what, in \cite{loday1998cyclic}, is called the \emph{naive} Hochschild homology. This does not necessarily agree with the ``correct'' version of Hochschild homology. However, in our applications, all algebras will be \emph{(bornologically)~$H$-unital}, a term we define later on, which suffices for both notions of Hochschild homology to coincide.
\end{remark}
\begin{definition}
\label{DefinitionConnesComplex}
Let~$A$ be a topological algebra. The \emph{bornological Connes complex} of~$A$ is given by
\begin{align*}
C_\bullet^{\lambda,\born}(A) := \bigoplus_{n \geq 0} C_n^{\lambda,\born}(A), \quad
C_n^{\lambda,\born}(A) := HC_n^\born(A,A)_{\mathbb{Z}/ (n+1)\mathbb{Z}},
\end{align*}
where the action of the generator~$\tau \in \mathbb{Z}/ (n+1)\mathbb{Z}$ on~$HC_n^\born(A)$ is given by cyclic permutation, meaning
\begin{align*}
\tau \cdot (a_0 \otimes \dots \otimes a_n) := (-1)^{n} a_n \otimes a_0 \otimes \dots \otimes a_{n-1} \quad \forall a_0,\dots,a_n \in A.
\end{align*}
The differential of this complex is induced by the Hochschild differential, which factors through to this complex.
The homology of this complex is called \emph{bornological cyclic homology}.
\end{definition}
Like in the algebraic setting, we will show the compatibility with an alternative definition of cyclic homology, in terms of the following double complex, see \cite[Chapter 2]{loday1998cyclic}:
\begin{definition}
\label{DefinitionCyclicBicomplex}
Let~$A$ be a topological algebra. The \emph{bornological cyclic double complex}~$CC^\born_{\bullet,\bullet}(A)$ of~$A$ is given by
\[
\begin{tikzcd}
	\dots
	\arrow{d}{b}  
	& 
	\dots 
	\arrow{d}{-b'}
	& 
	{\dots} 
	\arrow{d}{b} 
	& 
	{\dots} 
	\arrow{d}{-b'}
	& 
	{\dots} 
	\\
	{A^{\hotimes^3}}  
	\arrow{d}{b}  	
	& 
	{A^{\hotimes^3}} 
	\arrow{l}{1 - \tau} 
	\arrow{d}{-b'}  
	& 
	{A^{\hotimes^3}} 
	\arrow{l}{N} 
	\arrow{d}{b}  
	& 
	{A^{\hotimes^3}} 
	\arrow{l}{1 - \tau} 
	\arrow{d}{-b'}  
	& 
	{\dots} 
	\arrow{l}{N} 
	\\
	{A^{\hotimes^2}} 
	\arrow{d}{b}  
	& 
	{A^{\hotimes^2}} 
	\arrow{l}{1 - \tau} 
	\arrow{d}{-b'}  
	&
	{A^{\hotimes^2}} 
	\arrow{l}{N} 
	\arrow{d}{b}  
	& 
	{A^{\hotimes^2}} 
	\arrow{l}{1 - \tau} 
	\arrow{d}{-b'}  
	& 
	{\dots} 
	\arrow{l}{N} 
	\\
	{A} 
	& 
	{A} 
	\arrow{l}{1 - \tau}
	& 
	{A} 
	\arrow{l}{N}
	& 
	{A} 
	\arrow{l}{1 - \tau}
	& 
	{\dots}
	\arrow{l}{N} 
\end{tikzcd}
\]
Here:
\begin{itemize}
\item
$b$ is the Hochschild differential as in Definition~\ref{DefinitionHochschildComplex},
\item $b' : A^{\hotimes^{n+1}} \to A^{\hotimes^n}$ is the \emph{bar differential}, which is defined as the Hochschild differential without the last summand, so
\begin{align*}
b'(a_0 \otimes \dots \otimes a_n) :=&
\sum_{i=0}^{n-1} (-1)^{i} a_0 \otimes \dots \otimes a_i a_{i+1} \otimes \dots \otimes a_n,
\end{align*}
\item
~$\tau$ is the cyclic permutation as in Definition~\ref{DefinitionConnesComplex}, 
\item ~$N : A^{\hotimes^n} \to A^{\hotimes^n}$ is the \emph{norm operator}, defined as
\begin{align*}
N := 1 + \tau + \tau^2 + \dots + \tau^{n-1}.
\end{align*}
\end{itemize}
All squares anticommute, and we denote the total complex by~$CC_\bullet^\born(A)$, also called the \emph{bornological cyclic complex} of~$A$.
\end{definition}
Completely analogously one defines the \emph{cohomological} bornological Hoch\-schild, Connes, and cyc\-lic double complexes, by taking the continuous dual of all spaces in the above. We respectively denote them by
\begin{align*}
HC^\bullet_\born(A,A), \quad C^\bullet_{\lambda,\born}(A), \quad CC^{\bullet,\bullet}_\born(A).
\end{align*}
\begin{proposition}
\label{PropositionQuasiIsomorphismCyclicBicomplex}
Let~$A$ be a topological algebra over $\mathbb{R}$.
The canonical quotient
\begin{align*}
CC^\born_{\bullet}(A) \to C^{\lambda,\born}_\bullet(A)
\end{align*}
onto the zeroeth column and the canonical inclusion
\begin{align*}
C_{\lambda,\born}^\bullet(A) \to CC_\born^{\bullet}(A)
\end{align*}
into the zeroeth column are \emph{quasi-isomorphisms}, i.e. chain maps which reduce to isomorphisms of vector spaces on homology.
\end{proposition}
\begin{proof}
Both maps are straightforwardly chain maps, so it remains to show that they induce isomorphisms on (co-)homology.\\
We begin with a sketch of the algebraic, homological case, see \cite[Theorem 2.1.5]{loday1998cyclic}. There, one constructs homotopy operators~$h', h$ on every row of the algebraic double complex~$CC_{\bullet, \bullet}(A)$, fulfilling
\begin{align*}
h' N + (1-\tau)h = \id = N h' + h(1-\tau),
\end{align*}
showing that, all rows are acyclic in nonzero degree, and degree zero homology of the~$n$-th row equals the algebraic cyclic complex~$C_n^{\lambda}(A)$. This shows the algebraic case.
Now, all involved operators in \cite[Theorem 2.1.5]{loday1998cyclic} are easily seen to be continuous and hence extend to the respective topological completions, hence the homotopy argument extends to~$CC^\born_\bullet(A)$ and~$C^\bullet_{\lambda, \born}(A)$.
\\
The cohomological statement follows identically by dualizing the above homotopy operators.
\end{proof}
If~$A$ is additionally unital, there is another complex which calculates cyclic homology:
\begin{definition}
Let~$A$ be a unital topological algebra. The \emph{bornological~$(b,B)$-double complex}~$B_{\bullet,\bullet}^\born(A)$ of~$A$ is defined as
\[
\begin{tikzcd}
	\dots 
	\arrow{d}{b}
	& 
	\dots 
	\arrow{d}{b}
	& 
	\dots 
	\arrow{d}{b}
	\\
	{A^{\hotimes^3}} 
	\arrow{d}{b}
	& 
	{A^{\hotimes^2}} 
	\arrow{d}{b}
	\arrow{l}{B}
	& 
	A 
	\arrow{l}{B}
	\\
	{A^{\hotimes^2}} 
	\arrow{d}{b}
	& 
	A 
	\arrow{l}{B}
	\\
	A
\end{tikzcd}
\]
Here $b$ is the Hochschild differential and $B$ is the \emph{Connes operator}, defined as
\begin{align*}
B = (1-\tau)sN,
\end{align*}
with~$\tau$ the cyclic permutation, $N$ the norm operator (see Definition~\ref{DefinitionCyclicBicomplex}) and $s$ the \emph{extra degeneracy}, defined as
\begin{align*}
s : A^{\hotimes^n} \to A ^{\hotimes^{n+1}} , \quad  a_1 \otimes \dots \otimes a_n \mapsto 1 \otimes a_1 \otimes \dots \otimes a_n.
\end{align*}
All squares anticommute, and we denote the total complex by~$B_\bullet^\born(A)$.
\end{definition}
\begin{proposition}
\label{PropositionUnitalAlgebrasHomotopyEquivalence}
Let~$A$ be a unital topological algebra. Then the~$B_\bullet^\born(A)$ and~$CC_\bullet^\born(A)$ are continuously homotopy equivalent.
\end{proposition}
\begin{proof}
Similar to the proof of Proposition~\ref{PropositionQuasiIsomorphismCyclicBicomplex}, one builds homotopy operators by compositions of the continuous operators~$N, s, (1 - \tau)$, thus the well-known algebraic homotopy equivalence lifts to a topological one. See \cite[Proposition 3.8.1.]{khalkhali2009basic} for details.
\end{proof}
%
%
%
%
\subsection{Closed range theorems}
In the following we will need to understand whether the range of the differential of the cyclic complex is a topologically closed subspace of the complex. This condition makes the complex more well-behaved, in the situation of Fréchet spaces turning all differentials into topological homomorphisms, and allowing for the use of a Künneth formula. 
In this subsection, we will only consider Fréchet spaces, so we may always think of~$\hotimes$ as the closure of the projective tensor product due to Proposition~\ref{prop:PropertiesOfTopTensorProds}.
The following statement due to Serre is occasionally helpful:
\begin{theorem}\cite[Lemma 2]{serre1955theoreme}
\label{thm:SerreClosedRange}
Let~$d : A \to B$ be a continuous, linear map of Fréchet spaces. If~$d(A)$ is cofinite-dimensional in~$B$, then~$d(A)$ is closed and complemented.
\end{theorem}
We cite the following Theorem from \cite[Corollary 5.3]{gourdeau2005kuenneth}, a rudimentary version of which was already given in \cite{schwartz1953operations}.
\begin{theorem}
Let~$A, B$ be chain complexes of nuclear Fréchet spaces, bounded from below in the sense that
\begin{align*}
A_n = B_n = 0 \text{ if } n < 0.
\end{align*}
If the differentials of both complexes have closed range, we have an isomorphism of TVS
\begin{align*}
H_\bullet(A \hotimes B) \cong H_\bullet(A) \hotimes H_\bullet(B).
\end{align*}
\end{theorem}
Note that in the above situation,~$H_\bullet(A \hotimes B)$ is again canonically Fréchet, and thus the differential of~$A \hotimes B$ has closed range. Hence we can iterate this formula:
\begin{corollary}
\label{CorollarySchwartzKuenneth}
If~$A_1,\dots,A_n$ are finitely many complexes of nuclear Fréchet spaces, and all differentials have closed range, we have
\begin{align*}
H_\bullet \left( A_1 \hotimes \dots \hotimes A_n \right)
\cong
 H_\bullet \left(A_1 \right) \hotimes \dots \hotimes H_\bullet \left(A_n \right) .
\end{align*}
\end{corollary}
%
However, there does not seem to be a simple argument for why the Hoch\-schild/cyc\-lic differentials should have closed range in large generality.
In the cases that interest us, we can make use of the following easy proposition:
\begin{proposition}
\label{PropQuasiIsoClosedRange}
Let~$(C_\bullet, d_C),(D_\bullet,d_D)$ be com\-plex\-es of Hausdorff TVS, and let there be a continuous quasi-isomorphism~$\phi: C_\bullet \to D_\bullet$.
\begin{itemize}
\item[i)]
If~$d_D$ has closed range, then~$d_C$ has closed range.
\item[ii)] If~$C_\bullet$ and~$D_\bullet$ admit a continuous homotopy equivalence, then~$d_C$ has closed range if and only if~$d_D$ does.
\end{itemize}
\end{proposition}
\begin{proof}
\textit{i)}
Since the range of~$d_D$ is closed and fully contained in~$\ker d_D$, the projection 
\begin{align*}
\pi : \ker d_D \to \frac{\ker d_D}{\im d_D [-1]} = H_\bullet(D_\bullet) 
\end{align*}
is a continuous, linear map of TVS. Denote by~$\tilde{\phi}$ the continuous, linear map arising from the composition
\begin{align*}
\ker d_C 
\stackrel{\phi}{\to} 
\ker d_D 
\stackrel{\pi}{\to} 
H_\bullet(D_\bullet).
\end{align*}
Since~$\phi$ reduces to an isomorphism on homology and~$\im d_C \subset \ker \phi$, we have that the image of $d_C$ is equal to $\ker \tilde{\phi}$, hence~$\im d_C$ is the kernel of a continuous, linear map. Since all involved spaces are Hausdorff, this kernel is closed, and we are done.
\\
\textit{ii)}
This follows from i) since a continuous homotopy equivalence of chain complexes induces continuous quasi-isomorphisms in both directions.
\end{proof}
\begin{proposition}\cite[Theorem 26.3]{meise1997einfuhrung}
\label{PropositionClosedRangeIffDual}
If~$\phi : E \to F$ is a continuous linear map of Fréchet spaces, its range is closed if and only if its transpose~$\phi^* : F^* \to E^*$  has closed range, where~$E^*$ and~$F^*$ denote the respective strong duals of~$E$ and~$F$.
\end{proposition}
\begin{corollary}
\label{CorollaryAllCyclicComplexesAreSimultaneouslyClosed}
Let~$A$ be a Fréchet algebra.
\begin{itemize}
\item[i)]
The differential of the Connes complex~$C^{\lambda,\born}_\bullet(A)$ has closed range if and only if the differential of~$CC^\born_\bullet(A)$ has closed range. 
\item[ii)]
If~$A$ is unital, then all differentials of the complexes
\begin{align*}
C_\bullet^{\lambda,\born}(A), \quad CC^\born_\bullet(A), \quad B_{\bullet}^\born(A)
\end{align*}
have closed range if and only if a single one of them does.
\end{itemize}
\end{corollary}
\begin{proof}
\textbf{i)}
$\Longrightarrow$: If $C^{\lambda,\born}_\bullet(A)$ has closed range, then the quasi-isomorphism from Proposition~\ref{PropositionQuasiIsomorphismCyclicBicomplex} together with Proposition~\ref{PropQuasiIsoClosedRange} shows that~$CC^\born_\bullet(A)$ has closed range as well.
\\
$\Longleftarrow$: If~$CC^\born_\bullet(A)$ has closed range, then by Proposition~\ref{PropositionClosedRangeIffDual} so does the cohomological complex~$CC_\born^\bullet(A)$ since all spaces in the homological total complex are Fréchet and all differentials of the cohomological complex are induced by dualization. 
Again Propositions~\ref{PropositionQuasiIsomorphismCyclicBicomplex} and~\ref{PropQuasiIsoClosedRange} show that the cohomological complex~$C_{\lambda,\born}^\bullet(A)$ has closed range as well. By Proposition~\ref{PropositionClosedRangeIffDual} this extends to~$C^{\lambda,\born}_\bullet(A)$. The equivalence is shown.
\\
\textbf{ii)} This follows from i) and the continuous homotopy equivalence in Proposition~\ref{PropositionUnitalAlgebrasHomotopyEquivalence}.
\end{proof}
\subsection{The case of smooth functions}
The most important algebra in our context is~$A = C^\infty(M)$, the algebra of smooth functions on a smooth manifold~$M$ with its standard Fréchet structure, see for example \cite[Chapter 10.1]{treves1967topological} for details. The following theorem on its bornological/continuous Hochschild homology with coefficients in itself its well-established, see \cite{connes1985noncommutative} for the cohomological proof when~$M$ is compact, and \cite{pflaum1998continuous}, \cite{tele1998micro} for extensions to the noncompact case.
\begin{theorem}
\label{TheoremHochschildHomologyOfSmoothFunctions}
Let~$M$ be a smooth manifold (possibly with boundary). Equipping the Fréchet space of forms~$\Omega^\bullet(M)$ with the zero differential, the map
\begin{align*}
HC_\bullet^\born(C^\infty(M),C^\infty(M)) 
&\to 
\Omega^\bullet(M), \\
f_0 \otimes \dots \otimes f_n 
&\mapsto
f_0 d_\dR f_1 \wedge \dots \wedge d_\dR f_n \quad \forall f_i \in C^\infty(M)
\end{align*}
is a quasi-isomorphism.
\end{theorem}
The calculation of bornological cyclic homology of~$C^\infty(M)$ builds on the above fact and is also well-established. We will state it here together with the additional information that the differential has closed range.
\begin{theorem}
\label{thm:CyclicHomologyClosedRangeForSmoothFunctions}
Let~$A = C^\infty(M)$ with its standard Fréchet space structure. Then
\begin{align*}
H_n^{\lambda,\born}(C^\infty(M)) = \frac{\Omega^n(M)}{d_\dR \Omega^{n-1}(M)} \oplus H^{n-2}_\dR(M) \oplus H_\dR^{n-4}(M) \oplus \dots
\end{align*}
and the differential of the Connes complex~$C^{\lambda,\born}_\bullet(C^\infty(M))$ has closed range.
\end{theorem}
\begin{proof}
The first statement is proven, for example, in \cite[Example 3.10.1]{khalkhali2009basic}, and we sketch the proof. The isomorphism from Theorem~\ref{TheoremHochschildHomologyOfSmoothFunctions} induces continuous chain map from the bornological~$(b,B)$-double complex of~$C^\infty(M)$ to the following complex of Fréchet spaces:
\begin{equation}
\label{eq:ConstructedDoubleComplex}
\begin{tikzcd}
\dots  
\arrow{d}{0} 
& 
\dots  
\arrow{d}{0}   
& 
\dots  
\arrow{d}{0}  
\\
\Omega^2(M) 
\arrow{d}{0}   
& 
\Omega^1(M) 
\arrow{l}{d_\dR}
\arrow{d}{0}  
& 
\Omega^0(M)  
\arrow{l}{d_\dR}
\\
\Omega^1(M) 
\arrow{d}{0} 
& 
\Omega^0(M)
\arrow{l}{d_\dR}
& 
\\
\Omega^0(M)  
&  
&  \\
\end{tikzcd}
\end{equation}
This map is compatible with the differentials, and if we filter both double complexes by their columns and consider the induced spectral sequences, the map descends to an isomorphism on the first page. Thus, by the spectral sequence comparison theorem (see \cite[Theorem 5.2.12]{weibel1995introduction}), the original map is a quasi-isomorphism of the total complexes. Then the calculation of bornological cyclic homology of~$C^\infty(M)$ follows from the easy calculation of the total homology of \eqref{eq:ConstructedDoubleComplex}.
\cite[Proposition 5.4]{palamodov1972stein} states that the de Rham differential associated to a smooth manifold has closed range, hence also the total differential of the double complex \eqref{eq:ConstructedDoubleComplex}). Thus, by Proposition~\ref{PropQuasiIsoClosedRange}, so do the differentials of the total complex of the bornological~$(b,B)$-double complex, and by Corollary~\ref{CorollaryAllCyclicComplexesAreSimultaneouslyClosed} also the differential of~$C^\lambda_\bullet(A)$. This proves the statement.
\end{proof}
\subsection{The case of compactly supported functions}
\label{SubsectionHochschildAndCyclicOfCompactly}
In preparation for results about compactly supported gauge algebras, we want to establish some properties of algebras related to compactly supported functions. Specifically, recall that one generally defines the topology of~$C^\infty_c(\mathbb{R}^n)$ as a direct limit: 
For every compact~$K \subset \mathbb{R}^n$, define the Fréchet subalgebra
\begin{align*}
C^\infty_K(\mathbb{R}^n) := \{f \in C^\infty(\mathbb{R}^n) : \supp f \subset K \} \subset C^\infty(\mathbb{R}^n).
\end{align*}
Denote by~$\bar D_r(0) \subset \mathbb{R}^n$ the closed disk of radius~$r$ around 0. Then the inclusions~$\bar D_r(0) \subset \bar D_{r'}(0)$ for~$r' > r$ induce a direct system~$\{C^\infty_{\bar{D}_r(0)}(\mathbb{R}^n)\}_{r > 0}$ with
\begin{align*}
C^\infty_c(\mathbb{R}^n) = \varinjlim C^\infty_{\bar{D}_r(0)}(\mathbb{R}^n).
\end{align*}
All spaces within this direct system are Fréchet spaces, so if one equips~$C^\infty_c(\mathbb{R}^n)$ with the inductive limit topology, it is a strict LF-space. We will avoid working with this LF-space directly and just work with the underlying Fréchet spaces, since homology and~$\hotimes$ commute with direct limits on strict LF-spaces. 
However, we want to remark that the LF-algebra~$C^\infty_c(\mathbb{R}^n)$ itself also has good properties with respect to bornological homology theories, see \cite{meyer2010excision}.
From here on, set~$D := \bar{D}_1(0) \subset \mathbb{R}^n$. We first cite the following important factorization result from \cite[Theorem 3.4]{voigt1984factorization}:
\begin{theorem}
\label{TheoremVoigtFactorization}
The Fréchet algebra~$C^\infty_D(\mathbb{R}^n)$ has the \emph{bounded strong factorization property:} For every bounded set~$B \subset C^\infty_D(\mathbb{R}^n)$, there is a~$z \in C^\infty_D(\mathbb{R}^n)$, a continuous linear operator~$T : C^\infty_D(\mathbb{R}^n) \to C^\infty_D(\mathbb{R}^n)$ and a sequence~$\{\phi_n\}_{n \geq 1} \subset C^\infty_D(\mathbb{R}^n)$ with
\begin{align*}
z \cdot T(x) = x, \quad \quad T(x) = \lim_{n \to \infty} \phi_n \cdot x \quad  \forall x \in B.
\end{align*}
\end{theorem}
\begin{remark}
Note that in \cite[Def 1.1]{voigt1984factorization}, the definition of the strong factorization property is stated slightly different, and weaker in one particularly important aspect: They require that for every~$x \in B$, the element~$ T(x)$ must be contained in~$\overline{C^\infty_D(\mathbb{R}^n) \cdot x}$, the closed ideal generated by~$x$. 
This, too, would imply~$T(x) = \lim_{n \to \infty} \phi_n \cdot x$ for some~$\phi_n$, but the choice of~$\phi_n$ may depend on~$x$.
The stronger statement that the~$\phi_n$ can be chosen independent of~$x$ is also true, as seen in the proof \cite[Proposition 2.7]{voigt1984factorization} which is used to prove the above theorem. This goes unstated in our cited material, but is important not only for our application, but also for other users of this material such as \cite[Proposition 3.4]{ewald2004hochschild}, \cite[Theorem 7]{gong2004excision}, \cite[Theorem 6.1]{wodzicki1989excision}.
\end{remark}
\begin{corollary}
\label{CorollaryCompactlySupportedBarCycles}
For all~$f \in C^\infty_{D^k}((\mathbb{R}^n)^k)$, there are~$g \in C^\infty_D(\mathbb{R}^n)$,~$h \in C^\infty_{D^k}((\mathbb{R}^n)^k)$ with
\begin{align*}
f(x_1,\dots,x_k) = g(x_1) \cdot h(x_1,\dots,x_k) \quad \forall x_i \in D,
\end{align*}
so that if~$f$ in the above is a cycle in bar homology, so is~$h$.
\end{corollary}
\begin{proof}
Since~$C_{D^k}^\infty((\mathbb{R}^n)^k) \cong (C_D^\infty(\mathbb{R}^n))^{\hotimes^k}$, \cite[Thm 45.1]{treves1967topological} implies that we can represent every~$f \in C_{D^k}^\infty((\mathbb{R}^n)^k))$ as a certain series, i.e. there are null sequences
\begin{align*}
\{f_1^n\}_{n \geq 1},\dots,\{f_k^n\}_{n \geq 1} \subset C_{D}^\infty(\mathbb{R}^n)
\end{align*}
and a sequence of complex numbers~$\{\lambda_n\}_{n \geq 1}$ with~$\sum_{n=1}^\infty |\lambda_n| < 1$, so that
\begin{align*}
f = \sum_{n=1}^\infty \lambda_n f_1^n \otimes \dots \otimes f_k^n,
\end{align*}
where the sum is \emph{absolutely convergent} as a series of Fréchet space elements, meaning that for a generating sequence of seminorms~$\{p_i\}_{i \geq 0}$ of the topology of~$C^\infty_D(\mathbb{R}^n)$, the following real-valued series converges for all~$i_1,\dots,i_k \geq 0$:
\begin{align*}
\sum_{n=1}^\infty \lambda_n \cdot p_{i_1}(f_1^n) \dots p_{i_k}(f_k^n).
\end{align*}
Since the set~$\{f_1^n\}_{n \geq 1}$ is convergent, it is bounded, and we can apply Theorem~\ref{TheoremVoigtFactorization} to get a continuous operator~$T : C^\infty_D(\mathbb{R}^n) \to C^\infty_D(\mathbb{R}^n)$, a function~$g \in C^\infty_D(\mathbb{R}^n)$ and a sequence~$\{\phi_m\}_{m \geq 1} \in C^\infty_D(\mathbb{R}^n)$ with
\begin{align*}
f_1^n = g \cdot T f_1^n, \quad T f_1^n = \lim_{m \to \infty} \phi_m \cdot f_1^n \quad \forall n \in \mathbb{N}.
\end{align*}
By continuity of~$T$ the sequence
\begin{align*}
h := \sum_{n=1}^\infty \lambda_n (Tf_1^n) \otimes \dots \otimes f_k^n
\end{align*}
is also absolutely convergent and thus defines an element in~$C^\infty_{D^k}((\mathbb{R}^n)^k)$. Thus
\begin{align*}
f(x_1,\dots,x_n) = g(x_1) \cdot h(x_1,\dots,x_n) \quad \forall x_i \in \mathbb{R}^n.
\end{align*} 
Additionally, if~$f$ is a cycle in bar homology, then, since the bar differential~$b'$ is continuous and~$C^\infty_D(\mathbb{R}^n)$-linear with respect to multiplication in the first tensor argument, we have
\begin{align*}
\lim_{m \to \infty} b'((\phi_m f_1^n) \otimes \dots \otimes f_k^n) = \lim_{m \to \infty} \phi_m \cdot b'(f_1^n \otimes \dots \otimes f_k^n).
\end{align*}
But then~$b'(h) = \lim_{m \to \infty} \phi_m \cdot b'(f)$, so if~$f$ was a cycle in bar homology, so is~$h$. This concludes the proof.
\end{proof}
\begin{corollary}
\label{CorollaryBarComplexCompactAcyclic}
The bornological bar complex of~$C^\infty_D(\mathbb{R}^n)$ is acyclic.
\end{corollary}
\begin{proof}
In the notation of Corollary~\ref{CorollaryCompactlySupportedBarCycles}, every bar chain~$f \in C_k^{\barr}(C^\infty_D(\mathbb{R}^n))$ fulfils~$
f = b'(g \otimes h) + g \otimes b'(h)$, with~$b'(f) = 0$ implying~$b'(h) = 0$. This shows the statement.
\end{proof}
\begin{proposition}
\label{PropositionFlatHochschild}
In the notation of Appendix~\ref{AppendixFlatDeRhamPalamodov}, consider the space
\begin{align*}
\Omega^\bullet_\fl(D, \partial D) := \{ \omega \in \Omega^\bullet(D) : (j^\infty \omega) \att_{\partial D} = 0 \} \subset \Omega^\bullet(D).
\end{align*}
If~$\Omega^\bullet_\fl(D, \partial D)$ is equipped with the zero differential, then the quasi-isomorphism on the bornological Hochschild complex
\begin{align*}
HC_\bullet^\born(C^\infty(D)) \to \Omega^\bullet(D)
\end{align*}
from Theorem~\ref{TheoremHochschildHomologyOfSmoothFunctions} restricts to a quasi-isomorphism 
\begin{align*}
HC_\bullet^\born(C^\infty_D(\mathbb{R}^n)) \to \Omega^\bullet_\fl(D, \partial D).
\end{align*}
\end{proposition}
\begin{proof}
In the sense of \cite[Section 2]{brasselet2009teleman}, the bornological Hochschild complex of the subalgebra~$C^\infty_D(\mathbb{R}^n) \subset C^\infty(D \setminus \partial D)$ is \emph{locally isomorphic} to the bornological Hochschild complex of~$C^\infty(D \setminus \partial D)$. But then \cite[Prop 2.2]{brasselet2009teleman} shows that the bornological Hochschild homology of~$C^\infty_D(\mathbb{R}^n)$ is, via the above morphism, isomorphic to the subalgebra of differential forms of~$\Omega^\bullet(D \setminus \partial D)$ generated by the functions in~$C^\infty_D(\mathbb{R}^n)$. But this is exactly~$\Omega^\bullet_\fl(D, \partial D)$.
\end{proof}
We investigate this flat de Rham complex more in Appendix~\ref{AppendixFlatDeRhamPalamodov}. From the results there, we conclude the following:
\begin{theorem}
\label{TheoremContinuousCyclicHomologyOfCinftyD}
The differential of the bornological cyclic complex of~$C^\infty_D(\mathbb{R}^n)$ has closed range and
\begin{align*}
H^{\lambda,\born}_k(C^\infty_D(\mathbb{R}^n)) 
= 
\frac{\Omega^k_\fl(D, \partial D)}{d_\dR \Omega^{k-1}_\fl(D, \partial D)}
\oplus 
H^{k-2}_\sing(D, \partial D) 
\oplus 
H^{k-4}_\sing(D, \partial D) 
\oplus 
\dots
\end{align*}
where the relative homology denotes relative singular homology.
\end{theorem}
\begin{proof}
By Lemma~\ref{LemmaFlatDeRhamClosedAndComplemented}, the differential of~$\Omega^\bullet_\fl(D, \partial D)$ has closed range, thus the quotient map 
\begin{align}
\label{eq:QuotientMapInThmContinuousCyclicHom}
\Omega^\bullet_\fl(D, \partial D) \to \tfrac{\Omega^\bullet_\fl(D, \partial D)}{d \Omega^\bullet_\fl(D, \partial D)[-1]}
\end{align}
is a continuous map of Fréchet spaces. 
\\
Composing the quasi-isomorphism~$HC_\bullet^\born(C^\infty_D (\mathbb{R}^n)) \to \Omega^\bullet_\fl(D, \partial D)$ from Pro\-po\-si\-tion~\ref{PropositionFlatHochschild} with this quotient map allows one to construct a continuous map of double complexes from~$CC_{\bullet,\bullet}^\born(C^\infty_D(\mathbb{R}^n))$ to the following:
\begin{equation}
\label{DoubleComplexFlatDeRhamForCyclic}
\begin{tikzcd}
	\vdots & \vdots & \vdots & \vdots & \vdots & {\reflectbox{$\ddots$}} \\
	{\frac{\Omega_{\text{flat}}^2 (D, \partial D)}{d_{\dR} \Omega_{\text{flat}}^1(D, \partial D)}} & 0 & {H_\sing^2(D, \partial D)} & 0 & {H_\sing^2(D, \partial D)} & \dots \\
	{\frac{\Omega_{\text{flat}}^1 (D, \partial D)}{d_{\dR} \Omega_{\text{flat}}^0(D, \partial D)}} & 0 & {H_\sing^1(D, \partial D)} & 0 & {H_\sing^1(D, \partial D)} & \dots \\
	{\Omega_{\text{flat}}^0 (D, \partial D)} & 0 & {H_\sing^0(D, \partial D)} & 0 & {H_\sing^0(D, \partial D)} & \dots
	\arrow[from=3-2, to=3-1]
	\arrow[from=4-2, to=4-1]
	\arrow[from=2-2, to=2-1]
	\arrow[from=2-4, to=2-3]
	\arrow[from=3-4, to=3-3]
	\arrow[from=4-4, to=4-3]
	\arrow[from=2-3, to=2-2]
	\arrow[from=3-3, to=3-2]
	\arrow[from=4-3, to=4-2]
	\arrow[from=2-5, to=2-4]
	\arrow[from=3-5, to=3-4]
	\arrow[from=4-5, to=4-4]
	\arrow[from=4-6, to=4-5]
	\arrow[from=3-6, to=3-5]
	\arrow[from=2-6, to=2-5]
	\arrow[from=1-1, to=2-1]
	\arrow[from=2-1, to=3-1]
	\arrow[from=2-2, to=3-2]
	\arrow[from=3-2, to=4-2]
	\arrow[from=2-3, to=3-3]
	\arrow[from=3-3, to=4-3]
	\arrow[from=3-4, to=4-4]
	\arrow[from=2-4, to=3-4]
	\arrow[from=1-4, to=2-4]
	\arrow[from=1-3, to=2-3]
	\arrow[from=1-2, to=2-2]
	\arrow[from=3-1, to=4-1]
	\arrow[from=2-5, to=3-5]
	\arrow[from=3-5, to=4-5]
	\arrow[from=1-5, to=2-5]
\end{tikzcd}
\end{equation}
All differentials in \eqref{DoubleComplexFlatDeRhamForCyclic} are set to zero. By Lemma~\ref{LemmaFlatDeRhamClosedAndComplemented},~$H^k_\sing(D, \partial D)$ is either zero or equal to~$\frac{\Omega_\fl^n(D, \partial D)}{d_\dR \Omega_\fl^{n-1}(D, \partial D)}$ for all~$k \geq 0$, so the zero map and the quotient map~\eqref{eq:QuotientMapInThmContinuousCyclicHom} suffice to construct the indicated map of double complexes.
Calculate the spectral sequence of~$C_{\bullet,\bullet}^\born(C^\infty_D(\mathbb{R}^n))$ arising from filtration by columns. By Proposition~\ref{PropositionFlatHochschild} and Corollary~\ref{CorollaryBarComplexCompactAcyclic} the first page~$E^{\bullet,\bullet}_1$ contains~$\Omega_\fl^\bullet(D, \partial D)$ in even-numbered columns and~$0$ in odd-numbered columns, so we have~$E^{\bullet,\bullet}_1 = E^{\bullet,\bullet}_2$.
By unravelling the connecting homomorphisms using the homotopy equation in the proof of Corollary~\ref{CorollaryBarComplexCompactAcyclic} and the construction of the elements in the equation from Corollary~\ref{CorollaryCompactlySupportedBarCycles}, one can explicitly spell out the differential on the second page, showing that the nontrivial differentials~$\Omega_\fl^\bullet(D, \partial D) \to \Omega_\fl^{\bullet + 1}(D, \partial D)$ are equal to the de Rham differential. 
The cohomology of this flat, relative de Rham complex is equal to relative singular homology, see Lemma~\ref{LemmaFlatDeRhamClosedAndComplemented}. Hence the third page of the spectral sequence is exactly equal to the double complex \eqref{DoubleComplexFlatDeRhamForCyclic}. Hence the map of double complexes we constructed is an isomorphism between the third pages of the spectral sequences. 
By the spectral sequence comparison theorem \cite[Theorem 5.2.12]{weibel1995introduction}, this implies that the original map of double complexes is a quasi-isomorphism of the total complexes. This calculates the bornological cyclic homology as stated. 
Lastly, since the total differential of the double complex \eqref{DoubleComplexFlatDeRhamForCyclic} is zero, it is closed, and hence the differential of the total complex~$CC_{\bullet}^\born(C^\infty_D(\mathbb{R}^n))$ has closed range by Proposition~\ref{PropQuasiIsoClosedRange}. This concludes the proof.
\end{proof}
\begin{corollary}
\label{cor:CyclicHomologyOfCompactlySupported}
We have for all~$k \geq 0$
\begin{align*}
H_\bullet^{\lambda, \born}(C^\infty_c(\mathbb{R}^n)) 
\cong
\frac{\Omega^k_c(\mathbb{R}^n) }{d_\dR \Omega^{k-1}_c(\mathbb{R}^n)}
\oplus 
H^{k-2}_{\dR, c}(\mathbb{R}^n) 
\oplus 
H^{k-4}_{\dR, c}(\mathbb{R}^n) 
\oplus 
\dots
\end{align*}
where~$H^\bullet_{\dR, c}(\mathbb{R}^n)$ denotes compactly supported de Rham cohomology of~$\mathbb{R}^n$.
\end{corollary}
\begin{proof}
This follows from Theorem~\ref{TheoremContinuousCyclicHomologyOfCinftyD} by taking the direct limit~$C^\infty_c(\mathbb{R}^n) = \varinjlim C^\infty_{\overline{D}_r(0)}(\mathbb{R}^n)$, since homology, the cyclic action, and the bornological tensor product commute with strict direct limits, and
\begin{align*}
\frac{\Omega^k_c(\mathbb{R}^n) }{d_\dR \Omega^{k-1}_c(\mathbb{R}^n)}
=
\varinjlim
\frac{\Omega^k_\fl(\bar{D}_r, \partial \bar{D}_r(0))}{d_\dR \Omega^{k-1}_\fl(\bar{D}_r(0), \partial \bar{D}_r(0))}.
\end{align*}
\end{proof}
\begin{remark}
\label{RemarkCyclicHomologyOfDirectSum}
In preparation for what follows, we want to note that the above methods generalizes to finite, topological direct sums of copies of~$C^\infty_D(\mathbb{R}^n)$, meaning
\begin{align*}
H_\bullet^{\barr, \born}\left(\bigoplus_{i=1}^r C^\infty_D(\mathbb{R}^n) \right)  = 0,\quad
H_k^{\lambda,\born} \left(\bigoplus_{i=1}^r C^\infty_D(\mathbb{R}^n) \right) 
\cong
\bigoplus_{i=1}^r H_k^\lambda \left(C^\infty_D(\mathbb{R}^n) \right).
\end{align*}
\end{remark}
\section{Loday-Quillen-Tsygan theorems for Fréchet algebras}
In this section, we want to extend the classical Loday-Quillen-Tsygan theorem to the setting of certain topological algebras. 
For a detailled exposition of the algebraic Loday-Quillen-Tsygan theorem, we direct the reader to \cite[Chapter 9 \& 10]{loday1998cyclic}, or, for an abridged version of the relevant details, Appendix~\ref{AppendixAlgebraicLodayQuillen}.
\label{SectionLodayQuillenFrechet}
\subsection{A general topological LQT-Theorem}
\begin{definition}
Let~$A$ be a topological algebra and~$n \in \mathbb{N}$. Consider the topological Lie algebra
\begin{align*}
\gl_n(A) := \gl_n(\mathbb{K}) \hotimes A
\end{align*}
whose Lie bracket is given by
\begin{align*}
[g \otimes a, h \otimes b] := [g,h]_{\gl_n(\mathbb{K})} \otimes (a \cdot b) \quad \forall g,h \in \gl(\mathbb{K}), \, a,b \in A.
\end{align*}
For~$m \geq n$, the inclusions~$\gl_n(A) \to \gl_m(A)$ define a direct system of Lie algebras, and we can define
\begin{align*}
\gl(A) := \gl_\infty(A) := \varinjlim \gl_n(A)
\end{align*}
\end{definition}
\begin{remark}
Note that all~$\gl_n(A)$ for~$1 \leq n < \infty$ are finite-dimensional, so we can equivalently write
\begin{align*}
\gl_n(\mathbb{K}) \hotimes A
=
\gl_n(\mathbb{K}) \otimes_\beta A
=
\gl_n(\mathbb{K}) \otimes A.
\end{align*}
If~$A$ is Fréchet, then all~$\gl_n(\mathbb{K}) \otimes A$ are Fréchet, so~$\gl(A)$ is a strict LF-space. 
In this case, since the bornological tensor product on strict LF-spaces is compatible with inductive limits, we have 
\begin{align*}
\gl(A) = \varinjlim (\gl_n(\mathbb{K}) \hotimes A) \cong \gl(\mathbb{K}) \hotimes A.
\end{align*}
\end{remark}
\begin{definition}
Let~$\mathfrak{g}$ be a topological Lie algebra. We define \emph{bornological Lie algebra homology} of this~$\mathfrak{g}$ to be the homology of the \emph{bornological Chevalley-Eilenberg chain complex}
\begin{align*}
C^\born_\bullet(\mathfrak{g})
:=
\left(
\mathbb{K}
\stackrel{0}{\leftarrow}
\hat{\Lambda}^1 \mathfrak{g}
\stackrel{d}{\leftarrow}
\hat{\Lambda}^2 \mathfrak{g}
\stackrel{d}{\leftarrow}
\dots \right),
\end{align*} where~$\Lambda^k \mathfrak{g}$ denotes the coinvariants of~$\otimes^k \mathfrak{g}$ with respect to the action of the symmetric group~$\Sigma_k$ by antisymmetrization, and the hat denotes completion in the bornological tensor product.
The differential is the extension of the Chevalley-Eilenberg differential to the completion:
\begin{align*}
d (g_1 \wedge \dots \wedge g_n) := \sum_{i < j} (-1)^{i + j -1} [g_i, g_j] \wedge g_1 \wedge \dotsb \hat{g}_i \dotsb \hat{g}_j \dotsb \wedge g_n
\quad \forall g_i \in \mathfrak{g}.
\end{align*}
\end{definition}
%
%
%
\begin{proposition}
\label{prop:ReduceToCoinvariantsPossible}
For every~$n \in \mathbb{N}$ and topological algebra~$A$, the space~$\gl_n(A)$ admits an action by~$\gl_n(\mathbb{K})$. If~$A$ is unital, then the reduction to coinvariants
\begin{align*}
C_\bullet^\born(\gl_n(A)) \to
C_\bullet^\born(\gl_n(A))_{\gl_n(\mathbb{K})}
\end{align*}
is a quasi-isomorphism. 
\end{proposition}
\begin{proof}
Clearly,~$\gl_n(\mathbb{K})$ acts on~$\gl_n(A) \cong \gl_n(\mathbb{K}) \otimes A$ by a tensor product of the adjoint action and the trivial action. 
Unitality of~$A$ implies that~$\gl_n(\mathbb{K}) \subset \gl_n(A)$ exists as a subalgebra, and hence~$\gl_n(\mathbb{K})$ acts on the~$\gl_n(A)$-cochains.
By Proposition~\ref{prop:PropertiesOfTopTensorProds}, the completed bornological tensor product is associative and commutative, and if one of its factors is finite-dimensional, it agrees with the algebraic tensor product as a vector space. Hence:
\begin{align*}
C_k^\born(\gl_n(A)) \cong
\left(  \gl_n(\mathbb{K})^{\otimes^k}
\otimes 
 A^{\hotimes^k} \right)_{\Sigma_k}.
\end{align*}
We know that since~$\gl_n(\mathbb{K})$ acts trivially on~$A$, and the finite-dimensional tensor module~$\gl_n(\mathbb{K})^{\otimes^k}$ is completely reducible, hence~$C_k^\born(\gl_n(A))$ is completely reducible. 
From here on, the proof is essentially identical to the proof in algebraic setting, see \cite[Prop 10.1.18]{loday1998cyclic}. Complete reducibility gives us
\begin{align*}
C_\bullet^\born(\gl_n(A)) = C_\bullet^\born(\gl_n(A))_{\gl_n(\mathbb{K})} \oplus L_\bullet,
\end{align*}
where~$L_\bullet$ is a direct sum of simple modules, all of which~$\gl_n(\mathbb{K})$ acts nontrivially on. Since the~$\gl_n(\mathbb{K})$-action commutes with the Lie algebra differential, this is even a decomposition into subcomplexes.
We can further decompose~$L_\bullet = Z_\bullet \oplus K_\bullet$, where~$Z_\bullet = L_\bullet \cap \ker d$, and~$K_\bullet$ is a module-theoretic complement of~$Z_\bullet$ in~$L_\bullet$ so that both~$Z_\bullet, K_\bullet$ are direct sums of simple modules. Clearly,~$H_\bullet(L_\bullet) = H_\bullet(Z_\bullet)$.
A simple~$\gl_n(\mathbb{K})$-module~$M \subset Z_\bullet$ is generated by any of its elements that~$\gl_n(\mathbb{K})$ acts nontrivially on. As a consequence, every element of~$L_\bullet$ is of the form~$X \cdot c \in Z_\bullet$ for some cochain~$c \in Z_\bullet$ and~$X \in \gl(\mathbb{K})$.
We further have, for all~$X \in \gl_n(\mathbb{K})$ and~$c \in C_\bullet(\gl_n(A))$ the homotopy equation
\begin{align*}
X \cdot c = d (X \wedge c) + X \wedge dc.
\end{align*}
Hence every element of~$Z_\bullet$ is of the form~$X \cdot c = d (X \wedge c)$, so a boundary. This shows that~$Z_\bullet$ and~$L_\bullet$ are acyclic, and the proof is done.
\end{proof}
\begin{corollary}
If~$A$ is a unital Fréchet algebra, then
\begin{align*}
H_\bullet^\born(\gl(A)) 
\cong 
H_\bullet \left( \bigoplus_k \left( \mathbb{K}[\Sigma_k]  \otimes A^{\hotimes^k} \right)_{\Sigma_k} \right).
\end{align*}
\end{corollary}
\begin{proof}
The tensor product~$\hotimes$ commutes with direct limits on LF-spaces by Proposition~\ref{prop:PropertiesOfTopTensorProds}
\begin{align*}
C_k^\born(\gl(A)) 
&=
\left(  \gl (\mathbb{K})^{\otimes_\beta^k} 
\otimes_\beta 
 A^{\hotimes^k}  \right)_{\Sigma_k}
\\&=
\varinjlim 
\left(  \gl_n (\mathbb{K})^{\otimes^k} 
\otimes
 A^{\hotimes^k}  \right)_{\Sigma_k}
=
\varinjlim
C_k^\born(\gl_n(A)).
\end{align*}
By Proposition~\ref{prop:ReduceToCoinvariantsPossible}, the reduction to coinvariants
\begin{align*}
C_\bullet^\born(\gl_n(A)) \to
C_\bullet^\born(\gl_n(A))_{\gl_n(\mathbb{K})}
\end{align*}
is a quasi-isomorphism for all~$n \in \mathbb{N}$.
From the case~$A = \mathbb{K}$ considered in Propostion~\ref{PropositionInvariantTheoryOfGLn}, we get, for~$n \geq k$, the isomorphism
\begin{align*}
C_\bullet^\born(\gl_n(A))_{\gl_n(\mathbb{K})}
\cong
\left(  \left(\gl_n (\mathbb{K})^{\otimes^k}\right)_{\gl_n(\mathbb{K})}
\otimes 
 A^{\hotimes^k}  \right)_{\Sigma_k}
\cong \left( \mathbb{K}[\Sigma_k]  \otimes A^{\hotimes^k} \right)_{\Sigma_k}.
\end{align*}
Finally, since homology commutes with direct limits, and as every graded component of~$C_\bullet^\born(\gl_n(A))_{\gl_n(\mathbb{K})}$ becomes constant at some point in the direct limit, we get the following chain of isomorphisms:
\begin{align*}
H_\bullet^\born(\gl(A)) 
&\cong 
\varinjlim H_\bullet^\born(\gl_n(A)) 
\\&
\cong
\varinjlim H_\bullet( C_\bullet^\born(\gl_n(A))_{\gl_n(\mathbb{K})} )
\cong
H_\bullet \left( \bigoplus_k \left( \mathbb{K}[\Sigma_k]  \otimes A^{\hotimes^k} \right)_{\Sigma_k} \right).
\end{align*}
\end{proof}
Hence we may work with the latter complex instead, which is in some sense simpler: Since all~$\Sigma_k$ are finite groups, it is a complex of Fréchet spaces. A drawback is that the differential on this subcomplex is more difficult to describe. Regardless, we have the following:
\begin{proposition}
\label{PropositionExtensionOfThetaBijective}
Let~$A$ be a Fréchet algebra. The isomorphism of chain complexes 
\begin{align*}
\theta : 
\Lambda^\bullet C^\lambda_{\bullet -1}(A) 
\to 
\bigoplus_{k \in \mathbb{N}}
\left( \mathbb{K}[\Sigma_{k}] \otimes A^{\otimes^{k}} \right)_{\Sigma_{k}}
\end{align*} 
from Proposition~\ref{PropositionThetaMapLodayQuillen} extends to a continuous isomorphism of chain complexes
\begin{align*}
\hat \theta:
\hat \Lambda^\bullet C_{\bullet-1}^{\lambda,\born}(A) 
\to 
\bigoplus_{k \in \mathbb{N}}
\left( \mathbb{K}[\Sigma_{k}] \otimes A^{\hotimes^{k}} \right)_{\Sigma_{k}}.
\end{align*}
\end{proposition}
\begin{proof}
By definition of~$\theta$, we can decompose 
\begin{align*}
\Lambda^\bullet C_{\bullet -1}^\lambda(A) = \bigoplus_{k \geq 1} \bigoplus_{[\sigma] \subset \Sigma_k} Z_{[\sigma]},
\end{align*} 
where the direct sum over~$[\sigma]$ is carried out over all conjugacy classes~$ [\sigma] \in \Sigma_k$, and~$Z_{[\sigma]}$ the span of all elements~$[u_1] \wedge \dots \wedge [u_r]  \in \Lambda^\bullet C_{\bullet -1}^\lambda(A)$ with
\begin{align*}
\exists \tau \in [\sigma]: \theta ([u_1] \wedge \dots \wedge [u_r]) = \left[ \tau \otimes \left(u_1 \otimes \dots \otimes u_r \right) \right].
\end{align*}
For each~$k \geq 1$ and conjugacy class~$[\sigma] \subset \Sigma_k$, choose a representative~$\sigma$ in cycle decomposition
\begin{align*}
\sigma = (1 \; \dotsb \; k_1) \circ (k_1 + 1 \; \dotsb \; k_2) \circ \dots \circ  (k_{r-1} + 1 \; \dotsb \; k_r).
\end{align*}
Then, on~$Z_{[\sigma]}$, the map~$\theta$ arises from the continuous map
\begin{align*}
 A^{\otimes^k} \to \mathbb{K}[\Sigma_k] \otimes A^{\otimes^k}, \quad
a_1 \otimes \dots \otimes a_k \mapsto \sigma \otimes a_{1} \otimes \dots \otimes a_{k},
\end{align*} 
by composition with the quotient map
\begin{align*}
\mathbb{K}[\Sigma_k] \otimes A^{\otimes^k} \to \left( \mathbb{K}[\Sigma_k] \otimes A^{\otimes^k} \right)_{\Sigma_k},
\end{align*}
factoring through the kernel, and then restricting to the direct summand~$Z_{[\sigma]}$ in the kernel. These actions leave continuity invariant, so~$\theta$, when restricted to the direct summand~$Z_{[\sigma]}$, inherits continuity. But then $\theta$ itself is continuous.
Also, for every conjugacy class~$[\sigma] \subset \Sigma_k$, the restriction~$\theta \att_{Z_{[\sigma]}} : Z_{[\sigma]} \to \theta(Z_{[\sigma]})$ admits a continuous inverse, induced in the same way by a map
\begin{align*}
\mathbb{K}[[\sigma]] \otimes A^{\otimes^k} \to  A^{\otimes^k} , \quad
(\tau^{-1} \sigma \tau) \otimes a_1 \otimes \dots \otimes a_k \mapsto  a_{\tau(1)} \otimes \dots \otimes a_{\tau(k)}.
\end{align*}
The above assignment defines a continuous map due to the finiteness of the conjugacy class~$[\sigma]$. Since both~$\theta$ and its inverse are continuous, they extend to continuous maps between the completions of their respective domains and codomains, and these extensions still compose to the identity. 
Thus, these extensions are isomorphisms of chain complexes, proving the statement.
\end{proof}
\begin{corollary}
Let~$A$ be a unital Fréchet algebra. There is a continuous quasi-isomorphism
\begin{align*}
C_\bullet^\born(\gl(A)) \to \hat \Lambda^\bullet C_{\bullet -1}^{\lambda,\born}(A).
\end{align*}
\end{corollary}
Corollary~\ref{CorollarySchwartzKuenneth} then finally implies:
\begin{theorem}
\label{thm:NuclearFrechetAlgebrasFulfilLQT}
Let~$A$ be a nuclear unital Fréchet algebra, and assume that the differential of the bornological cyclic complex~$C^{\lambda,\born}_\bullet(A)$ has closed range. 
Then we have, for all~$r,n \in \mathbb{N}$ with~$r + 1 \leq n$
\begin{align*}
H_r^\born(\gl_n(A)) \cong  \left(\hat\Lambda^\bullet H_{\bullet -1}^{\lambda,\born}(A)\right)_r,
\end{align*}
and
\begin{align*}
H_\bullet^\born(\gl(A)) \cong \hat \Lambda^\bullet H_{\bullet -1}^{\lambda,\born}(A).
\end{align*}
\end{theorem}
\begin{remark}
Note that the only place where it mattered that we used~$\gl(A)$ rather than any of the other limits of classical simple Lie algebras~$\shl(A), \, \shp(A)$ or~$\so(A)$ was in the coinvariants for the tensor modules~$\gl(\mathbb{K})^{\otimes^k}$. 
We will not explicitly present this, but we do want to remark that one can gain statements analogous to Theorem~\ref{thm:NuclearFrechetAlgebrasFulfilLQT} for these other Lie algebras with very little modification, save that one occasionally may need to replace cyclic homology with the closely related \emph{dihedral homology}, which we have not defined here. We direct the reader to \cite[Chapter 9 \& 10]{loday1998cyclic} for a detailled discussion of the algebraic setting.
\end{remark}
Lastly, these results extend to certain non-unital algebras as well. Specifically, in the algebraic setting one can weaken the assumption of unitality to \emph{H-unitality}, a property which is defined as the acyclicity of the algebraic bar complex of~$A$, see \cite{hanlon1988complete} for a proof in the finite-dimensional setting and the preprint \cite{cortinas2005cyclic} (supplementing the publication \cite{cortinas2006obstruction}) for an explicit generalization to infinite-dimensional algebras. 
We delay the lengthy proof to Section~\ref{sec:ProofOfNonUnitalLQT}:
\begin{theorem}
\label{TheoremNonUnitalLQT}
Let~$A$ be a nuclear Fréchet algebra, and assume that the differential of the bornological cyclic complex~$C^{\lambda,\born}_\bullet(A)$ has closed range. 
Additionally, assume that~$A$ is \emph{bornologically H-unital}, i.e. the bornological bar complex~$C^{\barr,\born}(A)$ is acyclic.
Then we have, for all~$r, n \in \mathbb{N}$ with~$2r + 1 \leq n$:
\begin{align*}
H_r^\born(\gl_n(A)) \cong \left(\hat \Lambda^\bullet H_{\bullet -1}^{\lambda,\born}(A)\right)_r,
\end{align*}
and
\begin{align*}
H_\bullet^\born(\gl(A)) \cong \hat \Lambda^\bullet H_{\bullet -1}^{\lambda,\born}(A)
\end{align*}
\end{theorem}
We give a rough sketch of the proof here: The complex~$C_\bullet^\born(\gl_n(A))$ can be decomposed into isotypic components with respect to the action of~$\gl_n(\mathbb{K})$. In the unital case, we have seen that only the invariant component contributes to homology.
In a similar vein to the proof of Proposition~\ref{PropositionExtensionOfThetaBijective}, one constructs a morphism of every component to a certain complex involving combinations of the bar complex and the Connes complex of~$A$, see the construction of~$\phi$ in \cite[Theorem 3.1]{cortinas2005cyclic}. As one takes the direct limit~$n \to \infty$, these morphisms become stable isomorphisms. In every isotypic component in which the bornological bar complex appears nontrivially, the acyclicity of the bornological bar complex forces the whole component to become acyclic; this leaves only the invariant component to contribute to homology, and this component is related to the cyclic complex exactly as in the unital setting. 
%
\subsection{Application to~$A = C^\infty(M)$ and~$C^\infty_c(\mathbb{R}^n)$}
We now apply the results of the previous section to the case when~$A$ equals some spaces of smooth functions.
\begin{corollary}
\label{cor:LQTForSmoothFunctions}
Let~$M$ be a smooth manifold, we have
\begin{align*}
H_\bullet^\born(\gl(C^\infty(M))) 
&\cong 
\hat \Lambda^\bullet H_\bullet^{\lambda,\born}(C^\infty(M)), \\
H^\born_\bullet(\gl(C^\infty_c(\mathbb{R}^n)))
&\cong 
\hat \Lambda^\bullet H_\bullet^{\lambda,\born}(C^\infty_c(\mathbb{R}^n)).
\end{align*}
\end{corollary}
\begin{proof}
We have established in Theorem~\ref{thm:CyclicHomologyClosedRangeForSmoothFunctions} that the closed-range assumption of Theorem~\ref{thm:NuclearFrechetAlgebrasFulfilLQT} holds for the nuclear Fréchet algebra~$C^\infty(M)$, so the first isomorphism is shown.
Further, we have shown in Corollary~\ref{CorollaryBarComplexCompactAcyclic} and Theorem~\ref{TheoremContinuousCyclicHomologyOfCinftyD} that the~$C^\infty_{\bar{D}_r(0)}(\mathbb{R}^n)$ are bornologically H-unital and fulfil the closed-range assumption of Theorem~\ref{TheoremNonUnitalLQT}. This shows
\begin{align*}
H^\born_\bullet(\gl(C^\infty_{\bar{D}_r(0)}(\mathbb{R}^n))) 
= 
\hat{\Lambda}^\bullet H_{\bullet -1}^{\lambda,\born}(C^\infty_{\bar{D}_r(0)}(\mathbb{R}^n)).
\end{align*}
Now, homology and~$\hotimes$ commute with direct limits, and so we have
\begin{align*}
H^\born_\bullet(\gl(C^\infty_c(\mathbb{R}^n)) )
&=
\varinjlim  H^\born_\bullet(\gl(C^\infty_{\bar{D}_r(0)}(\mathbb{R}^n)))\\
&=
\varinjlim 
\hat{\Lambda}^\bullet H_{\bullet -1}^{\lambda,\born}(C^\infty_{\bar{D}_r(0)}(\mathbb{R}^n))
=
\hat \Lambda^\bullet H_\bullet^{\lambda,\born}(C^\infty_c(\mathbb{R}^n)).
\end{align*}
\end{proof}
%
%
%
%
%
\section{Bornological homology of nontrivial gauge algebras}
\label{SectionNontrivialGaugeAlgs}
Now, let~$M$ be a finite-dimensional, smooth manifold,~$H$ a finite-dimensional Lie group with associated Lie algebra~$\mathfrak{h}$, and~$P \to M$ a principal~$H$-bundle. Consider the \emph{adjoint bundle}~$\Ad P := P \times_{\Ad} \mathfrak{h}$.
We would like to understand the bornological Lie algebra homology of the compactly supported gauge algebra~$\Gamma_c(\Ad P )$. For small open sets~$U \subset M$, sections of~$\Ad P  \att_U \to U$ can be identified with~$\mathfrak{h} \otimes C^\infty(U)$. As a consequence, for simple classical Lie algebras~$\mathfrak{h}$, the methods of the previous section allow a calculation of the stable part of bornological homology~$H_\bullet^\born(\mathfrak{h} \otimes C^\infty_c(U))$.
By employing local-to-global methods as in the spectral sequence calculus developed in \cite{bott1977cohomology}, we will be able to deduce information about~$H^\born_\bullet(\Gamma_c(\Ad P ))$ from this local data. As a model of the more general case, we will only consider the case~$\mathfrak{h} = \mathfrak{gl}_n(\mathbb{K})$ for some~$1 \leq n < \infty$, as this connects directly to the methods from Section~\ref{SectionLodayQuillenFrechet}.
The methods of the following section are not exclusive to gauge algebras: In more generality, the following methods can be applied to calculate bornological Lie algebra homology of~$\Gamma_c(A)$ for any Lie algebroid~$A \to M$, assuming that the bornological Lie algebra homology of~$\Gamma_c(A \att_U)$ can be calculated whenever~$U$ is a small disk over which~$A$ trivializes.
\subsection{Cosheaves of Lie algebra chains and diagonal homology}
Fix in this subsection a smooth, locally trivial Lie algebra bundle~$\mathcal{K} \to M$ with fi\-nite\--di\-men\-sion\-al fibre~$\gl_n(\mathbb{K})$ for some~$1 \leq n < \infty$. Whenever we speak of (pre-)cosheaves in the following, we think of them as valued in the category of abelian groups as in Appendix~\ref{AppendixCosheavesOnABase}, but this is largely unimportant.
We first introduce some notation: Define for a compact~$K \subset M$ the closed Fréchet subspace
\begin{align*}
\Gamma_K(\mathcal{K}) := \{s \in \Gamma(\mathcal{K}) : \supp s \subset K \} \subset \Gamma(\mathcal{K}).
\end{align*}
Then, given any compact exhaustion~$\{K_n\}$ of~$M$, the topology on the compactly supported section space arises from the direct limit topology~$ \Gamma_c(\mathcal{K}) = \varinjlim \Gamma_{K_n}(\mathcal{K})$.
It is well-known that if we have two vector bundles~$A,B \to M$, then 
\begin{align}
\label{eq:ExteriorTensorProductOfSections}
\Gamma(A) \hotimes \Gamma(B) \cong \Gamma(A \boxtimes B),
\end{align}
where~$A \boxtimes B := \pr_1^* A \otimes \pr_2^* B \to M \times M$ denotes the exterior tensor product of vector bundles, see for example \cite[Thm 51.6]{treves1967topological} for the statement in trivial fibres.
Then, using Proposition~\ref{prop:PropertiesOfTopTensorProds}:
\begin{align*}
\Gamma_c(\mathcal{K})^{\hotimes^k}
\cong 
\varinjlim \Gamma_{K_n}(\mathcal{K})^{\hotimes^k}
\cong 
\varinjlim
\Gamma_{K_n \times \dots \times K_n}(\mathcal{K}^{\boxtimes^k})
\cong 
\Gamma_c(\mathcal{K}^{\boxtimes^k}).
\end{align*}
This justifies the following definition:
\begin{definition}
We define, for every~$k \geq 1$ the precosheaf~$B_k(\mathcal{K}, \cdot)$ over~$M^k$, assigning to an open set~$U \subset M^k$ the set
\begin{align*}
B_k(\mathcal{K}, U) := \Gamma_c(\mathcal{K}^{\boxtimes^k} \att_U).
\end{align*}
The precosheaf map~$\Gamma_c(\mathcal{K}^{\boxtimes^k} \att_U) \to \Gamma_c(\mathcal{K}^{\boxtimes^k} \att_V)$ associated to the inclusion~$U \subset V$ is defined via extension by zero.
\end{definition}
\begin{remark}
From this definition and the previous isomorphism, we get the bor\-no\-lo\-gi\-cal Lie algebra complex via restricting to global sections and~$\Sigma_k$-coinvariants, i.e. for all~$k \geq 1$ we have
\begin{align*}
C_k^\born(\Gamma_c(\mathcal{K})) 
\cong 
(B_k(\mathcal{K}, M^k))_{\Sigma_k}.
\end{align*}
Note in particular that we set the zero degree part to zero, so we think of~$B_k$ as a reduced complex.
\end{remark}
Since the precosheaf~$B_k(\mathcal{K},\cdot)$ arises from the compactly supported sections of a soft sheaf (given by the sections of a smooth vector bundle), Proposition~\ref{PropositionSoftSheavesGiveFlabbyCosheaves} implies:
\begin{lemma}
For every~$k \geq 1$, the precosheaf~$B_k(\mathcal{K},\cdot)$ over~$M^k$ is a flabby cosheaf.
\end{lemma}
\begin{definition}
Let~$k \geq 1$ and~$q \geq 0$ be integers. 
Define the \emph{$q$-th diagonal} in~$M^k$ via
\begin{align*}
M^k_q 
:= 
\{ (x_1,\dots,x_k) \in M^k : 
|\{x_1,\dots,x_k\}| \leq q\}.
\end{align*}
\end{definition}
\begin{remark}
The~$q$-th diagonals interpolate between the well-known notions of the \emph{thin} diagonal
\begin{align*}
M_\mathrm{thin}^k := M_1^k = \{(x,\dots,x) \in M^k\},
\end{align*}
and the \emph{fat} diagonal
\begin{align*}
M_\mathrm{fat}^k := M_{k-1}^k = \{(x_1,\dots,x_k) \in M^k : \exists i \neq j \text{ s.t. } x_i = x_j\}.
\end{align*}
We have a chain of inclusions:
\begin{align*}
\emptyset = M^k_0 \subset M^k_1 \subset \dots \subset M^k_{k-1} \subset M^k_k = M^k.
\end{align*}
\end{remark}
\begin{definition}
Let~$k, q \geq 1$ be integers. Define the precosheaf~$\Delta_q B_k(\mathcal{K}, \cdot)$ on~$M^k_q$ as follows: For every open set~$V \subset M^k_q$, define
\begin{align*}
\Delta_q B_k(\mathcal{K},W) :=  
\varprojlim_{U \supset W}
B_k(\mathcal{K},U) 
/ 
B_k(\mathcal{K}, U \cap (M^k \setminus M^k_q)).
\end{align*}
Here, the projective limit is taken over all open~$U \subset M^k$ containing~$W$, and the limit maps are induced by the extension maps~$\iota_U^V$ of the cosheaf~$B_k(\mathcal{K}, \cdot)$ for~$U \subset V$. The precosheaf maps are induced by the ones of~$B_k(\mathcal{K},\cdot)$.
\end{definition}
Equivalently,~$\Delta_q B_k(\mathcal{K},V)$ denotes the compactly supported sections of the sheaf of holonomic germs of~$\Gamma(\mathcal{K})$ along~$V \subset M^k_q$. Thus~$\Delta_q B_k(\mathcal{K},\cdot)$ equals the precosheaf of compactly supported sections of a soft sheaf, so Proposition~\ref{PropositionSoftSheavesGiveFlabbyCosheaves} implies:
\begin{lemma}
For all integers~$k, q \geq 1$, the precosheaf~$\Delta_q B_k(\mathcal{K},\cdot)$ on~$M^k_q$ is a flabby cosheaf.
\end{lemma}
\begin{definition}
\label{DefinitionDiagonalChains}
Let~$k, q \geq 1$ be integers, and~$U \subset M$ open.
We define a precosheaf~$\Delta_q C_k(\mathcal{K},\cdot)$ on~$M$, the \emph{$q$-diagonal~$k$-chains} of~$\Gamma(\mathcal{K})$:
\begin{align*}
\Delta_q C_k(\mathcal{K},U) := \left( \Delta_q B_k(\mathcal{K}, U^k \cap M^k_q) \right)_{\Sigma_k}.
\end{align*}
When $U  = M$, we also set the notation
\begin{align*}
\Delta_q C_k^\born(\Gamma_c(\mathcal{K})) 
:=
\Delta_q C_k(\mathcal{K},M) .
\end{align*}
\end{definition}
\begin{remark}
The~$q$-diagonal chains can be viewed as a quotient of the usual bornological complex, i.e. for every~$k,q \geq 1$ there is a natural projection
\begin{align*}
C_k^\born (\Gamma_c(\mathcal{K} \att_U)) 
\to 
\Delta_q C_k(\mathcal{K},U),
\end{align*}
and the Lie algebra differential factors through to a well-defined differential on this new complex.
\end{remark}
\subsection{Cosheaves of compactly supported differential forms}
\label{subsec:SmoothFormsAndKernelCosheaf}
Fix a smooth manifold~$M$ of dimension~$n$, and some~$k \in \mathbb{N}_{0}$ with~$0 \leq k \leq n$. 
\begin{definition}
\label{def:PrecosheavesOfCompactFormsAndZk}
Define the precosheaves~$\Omega^k_c$ and~$Z^k$, respectively given by assigning to an open~$U \subset M$ the compactly supported~$k$-forms~$\Omega^k_c(U)$ and
\begin{align*}
Z^k(U) := \frac{\Omega_c^k(U)}{d_\dR \Omega_c^{k-1}(U)}.
\end{align*}
The extension maps of~$\Omega^k_c$ are induced by the extension of compactly supported forms by zero. They induce the extension maps on the quotient~$Z^k$.
\end{definition}
We find:
\begin{lemma}
The precosheaves~$\Omega^k_c$ and~$Z^k$ over~$M$ are cosheaves. Further,~$\Omega^k_c$ is flabby and~$Z^k$ admits the flabby coresolution
\begin{align}
\label{eq:FlabbyResolutionOfZk}
0 \to \Omega^0_c \to \Omega^1_c \to \dots \to \Omega^k_c \to Z^k \to 0,
\end{align}
where the last nontrivial map is the canonical quotient map, and the other nontrivial maps are given by the de Rham differential.
\end{lemma}
\begin{proof}
The~$\Omega^k_c$ are flabby cosheaves by Proposition~\ref{PropositionSoftSheavesGiveFlabbyCosheaves}, since they arise as the precosheaf of compactly supported sections of the soft sheaf of differential forms on~$M$. 
The de Rham differential induces a cosheaf morphism~$\Omega^{k-1}_c \to \Omega^k_c$ whose cokernel precosheaf equals exactly~$Z^k$, and cokernels precosheaves of cosheaf morphisms are automatically cosheaves \cite[Chapter VI, Proposition 1.2]{bredon1997sheaf}. 
The Poincaré lemma for~$\Omega_c^\bullet(\mathbb{R}^n)$ implies that the sequence \eqref{eq:FlabbyResolutionOfZk} is locally exact, and hence it is a flabby coresolution for~$Z^k$. The statement is proven.
\end{proof}
The coresolution shows, together with Proposition~\ref{PropositionConnectionCechAndCoresolutionHomology}:
\begin{corollary}
\label{cor:CechHomologyOfZk}
The \v{C}ech homology of~$Z^k$ equals
\begin{align*}
\check{H}_r(M,Z^k)
=
\begin{cases}
Z^k(M) &\text{ if } r = 0, \\
H^{k - r}_{\dR,c}(M) & \text{ if } r > 0,
\end{cases}
\end{align*}
where~$H^\bullet_{\dR,c}(M)$ denotes compactly supported de Rham cohomology of~$M$.
\end{corollary}
We will be working with certain products of the above cosheaves over the cartesian products~$M^1,M^2,M^3,\dots$. Let us formalize what we mean by this:
\begin{lemma}
Let~$l \in \mathbb{N}$ and~$0 \leq k_1,\dots,k_l \leq n$. 
There is a cosheaf~$Z^{k_1} \hotimes \dots \hotimes Z^{k_l}$ over~$M^l$ with the property that for all open~$U_1, \dots, U_l \subset M$ we have
\begin{align*}
(Z^{k_1} \hotimes \dots \hotimes Z^{k_l})(U_1 \times \dots \times U_l) := Z^{k_1}(U_1) \hotimes \dots \hotimes Z^{k_l}(U_l),
\end{align*}
and the cosheaf map associated to an inclusion~$U_1 \times \dots \times U_l \subset V_1 \times \dots \times V_l$  equals the tensor product of the extension maps of the~$Z^{k_1},\dots,Z^{k_l}$.
\end{lemma}
\begin{proof}
For simplicity, we treat the case~$l = 2$, from which the general case easily follows. 
Consider the topological base~$\mathcal{B} := \{U \times V : U, V \subset M \text{ open}\}$ of~$M^2$, and the precosheaf~$\Omega^{k_1}_c \hotimes \Omega^{k_2}_c$ on~$\mathcal{B}$ given by
\begin{align*}
(\Omega^{k_1}_c \hotimes \Omega^{k_2}_c) (U \times V) := \Omega^{k_1}_c(U) \hotimes \Omega^{k_2}_c(V) \quad \forall U \times V \in \mathcal{B}.
\end{align*}
Similar to \eqref{eq:ExteriorTensorProductOfSections}, we have for all open~$U, V \subset M$ the isomorphism
\begin{align*}
\Omega^{k_1}_c(U) \hotimes \Omega^{k_2}_c(V) 
\cong 
\Gamma_c( \left(\Lambda^{k_1} T^*M \boxtimes \Lambda^{k_2} T^*M \right) \att_{U \times V}),
\end{align*}
which extends to an isomorphism of precosheaves on the base~$\mathcal{B}$. The right-hand side defines a cosheaf on~$M^2$ as the cosheaf of compactly supported sections of a vector bundle. Hence, it restricts to a cosheaf on the base~$\mathcal{B}$, and thus~$U \times V \mapsto \Omega^{k_1}_c(U) \hotimes \Omega^{k_2}(V)$ defines a cosheaf on the base~$\mathcal{B}$.
Consider the morphism of cosheaves on~$\mathcal{B}$ given by
\begin{align*}
d_{\dR} \hotimes \id 
:
\Omega^{k_1 - 1}_c \hotimes  \Omega^{k_2}_c
\to
\Omega^{k_1}_c \hotimes \Omega^{k_2}_c.
\end{align*}
If~$\Omega^\bullet_{K}(U)$ denotes the Fréchet space of differential forms on an open set~$U$ whose support is contained in a compact set~$K$, and~$\{K_n\}_{n \geq 1}$ and~$\{L_n\}_{n \geq 1}$ denote compact exhaustions of open sets~$U,V \subset M$, then the image of the morphism~$d \hotimes \id$ at the open set~$U \times V$ is
\begin{align*}
\varinjlim d_\dR \Omega^{k_1 -1}_{K_n}(U) \hotimes \Omega^{k_2}_{L_n}(V)
=
d_\dR \Omega^{k_1 -1}_{c}(U) \hotimes \Omega^{k_2}_{c}(V).
\end{align*}
The range of the de Rham differential~$d: \Omega^k_{K_n}(U) \to \Omega^{k+1}_{K_n}(U)$ is closed (c.f. Lemma~\ref{LemmaFlatDeRhamClosedAndComplemented}), so using Proposition~\ref{prop:ProjTensorProductExact} we deduce
\begin{align*}
\frac{\Omega^{k_1}_{c}(U) \hotimes \Omega^{k_2}_{c}(V)}
{d_\dR \Omega^{k_1 -1}_{c}(U) \hotimes \Omega^{k_2}_{c}(V)}
&\cong
\varinjlim 
\frac
{\Omega^{k_1}_{K_n}(U) \hotimes \Omega^{k_2}_{L_n}(V)}
{d_\dR \Omega^{k_1 -1}_{K_n}(U) \hotimes \Omega^{k_2}_{L_n}(V)}
\\
&\cong 
\varinjlim 
\frac
{\Omega^{k_1}_{K_n}(U)}
{d_\dR \Omega^{k_1 -1}_{K_n}(U)} 
\hotimes 
\Omega^{k_2}_{L_n}(V)
\cong
Z^{k_1}_c(U) \hotimes \Omega^{k_2}_c(V).
\end{align*}
By Proposition~\ref{prop:CokernelCosheaves} the cokernel of~$d_{\dR} \hotimes \id~$ defines a cosheaf on the base~$\mathcal{B}$, and, by the previous calculations, this cosheaf on~$\mathcal{B}$ assumes on~$U \times V \in \mathcal{B}$ the shape~$Z^{k_1}(U) \hotimes \Omega^{k_2}(V)$; we denote this cosheaf by~$Z^{k_1} \hotimes \Omega^{k_2}$.
Analogously by considering the cokernel of
\begin{align*}
Z^{k_1} \hotimes  \Omega^{k_2 - 1}_c
\stackrel{\id \hotimes d_{\dR}}{\to}
Z^{k_1} \hotimes \Omega^{k_2}_c
\end{align*}
we find the cosheaf~$Z^{k_1} \hotimes Z^{k_2}$ on the base~$\mathcal{B}$, which admits for all~$U \times V \in \mathcal{B}$ the desired local form 
\begin{align*}
(Z^{k_1} \hotimes Z^{k_2})(U \times V) = Z^{k_1}(U) \hotimes Z^{k_2}(V).
\end{align*}
Now the statement follows since a cosheaf on $\mathcal{B}$ extends uniquely to a cosheaf on $M$ by Theorem~\ref{thm:CosheavesOnBaseExtendToCosheavesOnSpace}.
\end{proof}
Unfortunately, we are currently not able to calculate the \v{C}ech homology of these product cosheaves. Let us remark the difficulties. Consider open covers~$\mathcal{U}, \mathcal{V}$ of~$M$, and the arising product cover~$\mathcal{U} \times \mathcal{V} := \{U \times V : U \in \mathcal{U}, V \in \mathcal{V}\}$ of~$M^2$. Then one can  deduce
\begin{align*}
\check{H}_\bullet(\mathcal{U} \times \mathcal{V}, Z^{k_1} \hotimes Z^{k_2})
\cong
\check{H}_\bullet(\mathcal{U}, Z^{k_1})
\hotimes
\check{H}_\bullet(\mathcal{V}, Z^{k_2}),
\end{align*}
using that the respective \v{C}ech complex for~$Z^{k_1} \hotimes Z^{k_2}$ factorizes, together with a direct limit argument and the Künneth formula from Theorem~\ref{CorollarySchwartzKuenneth}.
However, product covers~$\mathcal{U} \times \mathcal{V}$ do \emph{not} generally constitute a cofinal subsystem in the directed system of open covers of~$M^2$, and as a consequence, the \v{C}ech homology can not be deduced from the associated projective limit. Without this cofinality, it is unclear to the author how to arrive at a Künneth theorem in the topological setting, such as \cite{cassa1973formule} and \cite{kaup1967kuenneth} – see also \cite{kaup1968topologische}, where the non-cofinality of product covers is acknowledged. We state the likely Künneth formula for our setting as a conjecture.
\begin{conjecture}
\label{con:CechHomologyOfProduct}
If $\mathcal{U}$ is a cover of $M^l$ so that any intersection of elements in $U$ are diffeomorphic to a finite union of Euclidean spaces, then
\begin{align*}
\check{H}_\bullet(\mathcal{U}, Z^{k_1} \hotimes \dots \hotimes Z^{k_l})
\cong
\check{H}_\bullet(M, Z^{k_1}) 
\hotimes 
\dots 
\hotimes 
\check{H}_\bullet(M, Z^{k_l}).
\end{align*}
\end{conjecture}
\subsection{A globalizing \v{C}ech double complex}
We now use the tools developed so far to examine homology of nontrivial~$\gl_t(\mathbb{K})$-bundles. The idea of using local-to-global spectral sequences goes back to Bott, Segal, Gelfand and Fuks, who used this strategy to describe the continuous Lie algebra cohomology of vector fields, see \cite{bott1977cohomology} and \cite{gelfand1969cohomologies}. Our approach will transfer their method to our current setting, in the spirit of the recently developed theory of factorization algebras, see \cite{costello2016factorization}.
\begin{definition}
Let~$q \geq 1$ and~$\mathcal{U} := \{U_\alpha\}$ an open cover of~$M$. We define the \emph{$q$-diagonal double complex associated to~$\mathcal{U}$} as the following double complex
\[\begin{tikzcd}
	\vdots & \vdots \\
	{\bigoplus_\alpha \Delta_q C_2(\mathcal{K},U_\alpha)} & 
	{\bigoplus_{\alpha,\beta} \Delta_q C_2(\mathcal{K},U_\alpha \cap U_\beta)} & 
	\dots \\
	{\bigoplus_\alpha \Delta_q C_1(\mathcal{K},U_\alpha)} & 
	{\bigoplus_{\alpha,\beta} \Delta_q C_1(\mathcal{K},U_\alpha \cap U_\beta)} & 
	\dots
	\arrow[from=2-1, to=3-1]
	\arrow[from=3-2, to=3-1]
	\arrow[from=2-2, to=2-1]
	\arrow[from=2-2, to=3-2]
	\arrow[from=1-1, to=2-1]
	\arrow[from=1-2, to=2-2]
	\arrow[from=3-3, to=3-2]
	\arrow[from=2-3, to=2-2]
\end{tikzcd}\]
The horizontal differentials are induced by the \v{C}ech differentials associated to the cosheaves~$\Delta_q B_k(\mathcal{K},\cdot)$, and the vertical differentials arise from the Chevalley-Eilenberg differential.
\end{definition}
\begin{remark}
We disregard the zeroeth homology groups on purpose, since they do not behave quite as neatly in the sheaf-theoretic sense, and are only connected to the complex by a zero differential.
\end{remark}
It will prove insightful to calculate the spectral sequences associated to this double complex in certain cases. To this end, we borrow a special notion of open covers from the theory of homotopy sheaves and factorization algebras:
\begin{definition} \cite[Definition 2.9]{debrito2013manifold}
Fix~$q,n \geq 1$. 
Let~$M$ be a smooth manifold of dimension~$n$, and~$\mathcal{U}$ an open cover of~$M$. We say that~$\mathcal{U}$ is \emph{$q$-good} if:
\begin{itemize}
\item[i)] All intersections of elements in~$\mathcal{U}$ are diffeomorphic to the disjoint union of at most~$q$ copies of~$\mathbb{R}^n$.
\item[ii)] If~$\{x_1,\dots,x_q\} \in M$ is a collection of~$q$ points, there is some~$U \in \mathcal{U}$ containing all~$x_1,\dots,x_q$.
\end{itemize}
\end{definition}
\begin{remark}
Note that~$1$-good covers are exactly the well-known \emph{good covers}, i.e. open covers so that all intersections of elements of the cover are diffeomorphic to~$\mathbb{R}^n$.
\end{remark}
Using the existence of Riemannian metrics on smooth manifolds, one can always construct such a cover:
\begin{proposition}\cite[Proposition 2.10]{debrito2013manifold}
For all~$q \in \mathbb{N}$ and every smooth manifold~$M$, there exists a~$q$-good cover on~$M$.
\end{proposition}
The following statement is straightforward to prove and shows us the use of this notion:
\begin{lemma}
\label{LemmaExponentiationOfQGoodCovers}
If~$\mathcal{U}$ is a~$q$-good cover of~$M$, then the sets 
\begin{align*}
\mathcal{U}^i := \{U^i \subset M^i : U \in \mathcal{U}\}, \quad
\mathcal{U}^i_q := 
\{U^i \cap M^i_q : U \in \mathcal{U}\}
\end{align*}
are open covers of~$M^i$ and~$M^i_q$, respectively, for all~$i = 1,\dots, q$
\end{lemma}
\begin{remark}
Another reason the notion of~$q$-good covers will be useful to us: since all intersections of elements in a~$q$-good cover~$\mathcal{U}$ are disjoint unions of Euclidean spaces, the section spaces~$\Gamma_c(\mathcal{K})$ restricted to these open sets becomes a finite direct sum of trivial gauge algebras~$\gl_n(C^\infty_c(\mathbb{R}^n))$.
Thus the Lie algebra homology of~$\Gamma_c(\mathcal{K}\att_U)$ for such~$U$ can be approached with the methods of the previous sections.
\end{remark}
\begin{proposition}
For every~$q,r \geq 1$ and every~$q$-good cover~$\mathcal{U}$ of~$M$, the~$r$-th row of the~$q$-diagonal double complex associated to~$\mathcal{U}$ are acyclic in nonzero degree, and their zeroeth homology is~$\Delta_q C_r^\born(M)$.
\end{proposition}
\begin{proof}
Recall that 
\begin{align*}
\Delta_q C_r(\mathcal{K},U) =
\left( \Delta_q B_r(\mathcal{K}, U^r \cap M^r_q) \right)_{\Sigma_r}.
\end{align*}
Hence, the~$r$-th row of the double complex is equal to the antisymmetrized \v{C}ech complex of the cosheaf~$\Delta_q B_r(\mathcal{K},\cdot)$ with respect to the cover~$\mathcal{U}^r_q$, using the notation of Lemma~\ref{LemmaExponentiationOfQGoodCovers}. 
Taking coinvariants with respect to the finite group~$\Sigma_r$ is an exact functor, hence the calculation of the homologies of the rows reduces to calculating the \v{C}ech homology of the cosheaves~$\Delta_q B_r(\mathcal{K},\cdot)$ and taking coinvariants afterwards.
But then the statement is a direct consequence of the fact that all these cosheaves are flabby, and thus have trivial \v{C}ech homology, see Corollary~\ref{CorollaryFlabbyCosheavesTrivialCech}.
\end{proof}
\begin{corollary}
\label{CorollarySpectralSequenceConvergesToDiagonal}
The total complex of the~$q$-diagonal double complex associated to a~$q$-good cover~$\mathcal{U}$ has homology equal to diagonal homology
\begin{align*}
\Delta_q H_\bullet^\born(\Gamma_c(\mathcal{K})) 
:=
H_\bullet(\Delta_q C_\bullet^\born(\Gamma_c(\mathcal{K}))).
\end{align*}
\end{corollary}
Hence the spectral sequence associated to the horizontal filtration will give us information about the diagonal bornological homology, granted that we understand the homology of the restricted algebras on all~$U \in \mathcal{U}$, and granted that we understand the Cech homology associated to the precosheaves of the homology groups~$U \mapsto H_\bullet (\Delta_q C_\bullet(\mathcal{K},U))$.
For a simple presentation of the following spectral sequence, let us introduce a piece of notation. Fix some~$n \in \mathbb{N}$. Then we set, for all~$k \in \mathbb{N}_0$:
\begin{align}
\label{FormulaNotationXiNK}
\xi_n(k) := 
\min \{k, n + (k - n \mod 2)\},
\end{align}
in other words, the sequence~$\{\xi_n(k)\}_{k \geq 0}$ assumes the shape
\begin{align*}
0, 1, 2, \dots, n - 1, n, n + 1, n, n+1, n, \dots
\end{align*}
Then, due to the periodic nature of cyclic homology, we can rephrase Corollary~\ref{cor:CyclicHomologyOfCompactlySupported} in the following way:
\begin{align}
\label{eq:CyclicHomologyInXiNotation}
H_k^{\lambda, \born}(C^\infty(\mathbb{R}^n)) 
\cong
\frac
{\Omega_c^{\xi_n(k)}(\mathbb{R}^n) }
{d_{\dR} \Omega_c^{\xi_n(k) - 1}(\mathbb{R}^n)} \quad \forall k \geq 0.
\end{align}
\begin{theorem}
\label{TheoremSpectralSequenceForGammaK}
Let~$M$ be a manifold of finite dimension~$n$ and~$P \to M$ a principal~$GL_t(\mathbb{R})$-bundle. Let ~$ \Ad(P) \to M$ be the associated adjoint bundle, and~$q := \lfloor \tfrac{t - 1}{2} \rfloor$. Denote by~$Z^k$ for~$k \geq 0$ the cosheaves over~$M$ from Definition~\ref{def:PrecosheavesOfCompactFormsAndZk} for~$k \geq 0$.
There is a homological first-quadrant spectral sequence~$\{E^\bullet_{r,s}\}_{r,s \geq 0}$ which converges to the homology of the~$(q+1)$-diagonal complex of~$\Gamma_c(\Ad(P))$, i.e. with the notation of Definition~\ref{DefinitionDiagonalChains} we have the convergence
\begin{align*}
E_{r,s}^k \implies 
\Delta_{q+1} H_{r+s}^\born ( \Gamma_c(\Ad(P))).
\end{align*}
For~$r \geq 0$ and~$1  \leq s \leq q$, we can express the second page~$E^2_{r,s}$ with the notation~\eqref{FormulaNotationXiNK}:
\begin{align*}
E^2_{r,s} = 
\bigoplus_{k \geq 1}
\left(
\bigoplus_{s_1 + \dots + s_k = s}
\check{H}_{r_1}
\left(\mathcal{U}^k, Z^{\xi_n(s_1 - 1)} \hotimes \dots \hotimes Z^{\xi_n(s_k - 1)} \right)
 \right)_{\Sigma_k} \quad 
\end{align*}
\end{theorem}
\begin{remark}
Assuming Conjecture~\ref{con:CechHomologyOfProduct}, we find that, graded by its diagonals, the second page~$\{E^2_{r,s}\}_{r,s \geq 0}$ looks exactly like the compactly supported analogue of~$H^\bullet(\gl(C^\infty(M)) = \hat{\Lambda}^\bullet H^{\lambda, \born}_{\bullet - 1}(C^\infty(M))$ in total degree~$\leq  q$, all instances of~$C^\infty(M)$ and~$\Omega^\bullet(M)$ replaced by~$C^\infty_c(M)$ and $\Omega^\bullet_c(M)$, respectively. 
Thus, while it is not instantly recognizable, the second page of this spectral sequence is a quite reasonable double grading of what one would expect bornological Lie algebra homology of~$\gl_n(C^\infty_c(M))$ to be, intertwining the homological grading with a certain grading that encodes the ``locality'' of the data.
\end{remark}
\begin{proof}[Proof of Theorem~\ref{TheoremSpectralSequenceForGammaK}]
We consider the~$(q+1)$-th diagonal double complex associated to the cover~$\mathcal{U}$, and consider the spectral sequence arising by filtering along the columns. We know that this spectral sequence converges to the~$(q+1)$-th diagonal homology by Corollary~\ref{CorollarySpectralSequenceConvergesToDiagonal}. It remains to describe the second page.
The differential on the zeroeth page is simply given by taking the homology in vertical direction.
 If~$s \leq q$ and~$U \in \mathcal{U}$, then
\begin{align*}
H_s( \Delta_{q+1} C_\bullet(\Ad(P),U) ) \cong H_s^\born(\gl_t(C^\infty_c(U)))  \cong H_s^\born(\gl(C^\infty_c(U))).
\end{align*}
If further~$U \cong \bigsqcup_{i=1}^r \mathbb{R}^n$, then
\begin{align*}
\Gamma_c(\mathcal{K} \att_U) \cong \bigoplus_{i=1}^r \gl(C^\infty_c(\mathbb{R}^n)).
\end{align*} 
With Remark~\ref{RemarkCyclicHomologyOfDirectSum} and Equation \eqref{eq:CyclicHomologyInXiNotation} we get
\begin{align*}
H_\bullet^\born(\gl(C^\infty_c(U))) 
\cong
\hat{\Lambda}^\bullet 
\left( \bigoplus_{i=1}^r 
H_{\bullet-1}^{\lambda,\born}(C^\infty_c(\mathbb{R}^n)) \right)
\cong
\hat{\Lambda}^\bullet \left(Z^{\xi_n(\bullet - 1 )}(U) \right).
\end{align*}
%
%
%
%
In degrees, this translates to:
\begin{align}
\label{eq:IdentificationOfCohomologyWithForms}
H_s^\born(\Gamma_c(\mathcal{K} \att_U) ) 
\cong
\bigoplus_{k \geq 1} 
\left( \bigoplus_{s_1 + \dots + s_k = s}  
Z^{\xi_n(s_1 - 1)}(U) 
\hotimes \dots \hotimes 
Z^{\xi_n(s_k - 1)}(U) \right)_{\Sigma_{k}}.
\end{align}
This calculates the first page of the spectral sequence.
Now, recall that the horizontal differential of the double complex arises from a \v{C}ech differential. To understand how this differential acts on the first page, it suffices to understand, for open sets~$U \subset V \subset \mathbb{R}^n$ with~$U \cong V \cong \mathbb{R}^n$, the composition
\begin{equation}
\label{MapExtensionOfHomologyMaps}
\begin{aligned}
H_k^\born( \gl_t(C^\infty_c(U))) &\cong
H_k^\born(\Gamma_c(\mathcal{K} \att_U ) )
\\
&\to
H_k^\born(\Gamma_c(\mathcal{K} \att_V ) )
\cong 
H_k^\born( \gl_t(C^\infty_c(V))), 
\end{aligned}
\end{equation}
where the middle map is induced by the extension map
\begin{align*}
C_k^\born(\Gamma_c(\mathcal{K}  \att_U )) \to C_k^\born(\Gamma_c(\mathcal{K}  \att_V )).
\end{align*}
The differential of the first page is then given as a linear combination of such maps. 
Under the identification of~$H_k^\born( \gl(C^\infty_c(U)))$ with antisymmetrized tensor products of terms~$Z^k(U)$, the map \eqref{MapExtensionOfHomologyMaps} is induced by extensions~$\iota_U^V: Z^k(U) \to Z^k(V)$, up to an action by the transition function~$g_{UV} : U \cap V \to \GL_t(\mathbb{R})$ arising from the choice of local trivializations on~$U$ and~$V$.
Now recall that on~$\Ad(P)$, the transition functions act by the adjoint action of~$\GL_t(\mathbb{K})$. However, in the calculation of~$H_k^\iota( \gl_t(C^\infty_c(U)))$ we have seen that we may reduce to the ~$\gl_t(\mathbb{K})$ tensor invariants, which are equal to the~$\GL_t(\mathbb{K})$ tensor invariants by \cite[Lemma 9.2.5]{loday1998cyclic}. Thus the transition functions act trivially on this space.
Hence, under the isomorphism \eqref{eq:IdentificationOfCohomologyWithForms}, the map \eqref{MapExtensionOfHomologyMaps} can be identified with the appropriate direct sum of extension maps for the product cosheaves~$Z^{s_1 -1 } \hotimes \dots \hotimes Z^{s_k - 1}$. As a consequence, the rows on the first page of the spectral sequence can be identified with the skew-symmetrization of the direct sum of \v{C}ech complexes of these product cosheaves, with respect to the covers~$\mathcal{U},  \mathcal{U}^2, \dots, \mathcal{U}^q$.  This concludes the proof.
\end{proof}
Some easy consequences:
\begin{corollary}
\label{cor:StableCohomologyInSpecialCases}
Let~$P \to M$ be a principal~$GL_t(\mathbb{K})$-bundle and~$\dim M \geq 1$, and assume that Conjecture~\ref{con:CechHomologyOfProduct} holds.
\begin{itemize}
\item[i)] If~$t \geq 3$, then
\begin{align*}
H_1^\born(\Gamma_c(\Ad(P))) \cong \Omega^0_c(M).
\end{align*}
\item[ii)]
If~$t \geq 5$, then
\begin{align*}
H_2^\born(\Gamma_c(\Ad(P))) &\cong 
\left(\Omega^0_c(M)^{\hotimes^2}\right)_{\Sigma_2} 
\oplus 
\frac{\Omega^1_c(M)}{d_\dR \Omega^0_c(M)}.
\end{align*}
\item[iii)] If~$t \geq 7$, then
\begin{align*}
H_3^\born&(\Gamma_c(\Ad(P)) 
\cong \\
&
\left(\Omega^0_c(M)^{\hotimes^3}\right)_{\Sigma_3}
\oplus 
\left(\frac{\Omega^1_c(M)}{d_\dR \Omega^0_c(M)} \hotimes \Omega^0_c(M)\right)
\oplus 
\frac{\Omega^2_c(M)}{d_\dR \Omega^1_c(M)}
\oplus
H^0_c(M).
\end{align*}
\end{itemize}
\end{corollary}
\section{Comments and further outlook on the spectral sequence}
We want to dedicate this section to a large amount of comments that can be made about Theorem~\ref{TheoremSpectralSequenceForGammaK}, its assumptions and possible outlooks for future generalizations.
\\
Firstly, the only reason we needed to restrict to bundles~$\Ad(P) \to M$ was to make sure that the transition functions of the bundle can be chosen to act by inner automorphisms of the Lie algebra~$\gl_t(\mathbb{K})$. In a more general Lie algebra bundle, \emph{outer automorphisms} may show up, i.e. automorphisms which are not in the image of~$\Ad : GL_t(\mathbb{K}) \to \text{Aut}(\gl_t(\mathbb{K}))$. 
In this case, one may need to twist the cosheaf structure of the~$Z^k$ by a locally constant system depending on the bundle. Hence, by including this data, it should be a straightforward exercise to generalize Theorem~\ref{TheoremSpectralSequenceForGammaK} to more general~$\gl_t(\mathbb{K})$-bundles.
\\
Secondly, in contrast to Theorems~\ref{thm:NuclearFrechetAlgebrasFulfilLQT} and~\ref{TheoremNonUnitalLQT}, one cannot phrase Theorem~\ref{TheoremSpectralSequenceForGammaK} in terms of Lie algebra bundles with in\-fi\-nite\--di\-men\-sion\-al fibre~$\gl(\mathbb{K})$, since the local sections in this situation use the wrong tensor product: \cite[Theorem 44.1]{treves1967topological} implies
\begin{align*}
C^\infty_c(U, \gl(\mathbb{K})) \cong \overline{C^\infty_c(U) \otimes_\pi \gl(\mathbb{K})},
\end{align*}
which is not necessarily equal to~$C^\infty_c(U) \hotimes \gl(\mathbb{K}) = \gl(C^\infty_c(U))$, as in general~$\otimes_\beta \neq \otimes_\pi$ on LF-spaces.
\\
Thirdly, the reader may be tempted to consider the non-compactly supported analogue of Theorem~\ref{TheoremSpectralSequenceForGammaK}. Constructing this would, in some ways, be more standard, as this would rather use the \emph{sheaf} theory associated to the sheaf of sections of~$\Ad(P) \to M$, rather than the cosheaf theory of its compactly supported counterpart. Another advantage would be that the stabilization 
\begin{align*}
H_k^\born(\gl(C^\infty(\mathbb{R}^n))) \cong H_k^\born(\gl_n(C^\infty(\mathbb{R}^n)))
\end{align*}
 occurs already when~$k + 1 \leq n$, which is strictly stronger than the stabilization for~$C^\infty_c(\mathbb{R}^n)$ in Theorem~\ref{TheoremNonUnitalLQT}, occurring only when~$2k + 1 \leq n$.
 However, the crucial difference is that the arising double complex would not be of first-quadrant anymore, as it arises as a mixture of a cohomological \v{C}ech complex with a homological Lie algebra complex. Such double complexes are in many ways more inconvenient: For one, it is not immediately clear whether the spectral sequences associated to the horizontal and vertical filtration converge to the same infinity-page, and even if they do, the relevant diagonals of the spectral sequence may hit infinitely many nonzero terms. 
\\
In particular, knowing only a part of the second page~$E_2^{p,q}$, as it is the case in Theorem~\ref{TheoremSpectralSequenceForGammaK}, is then a lot less helpful, as it does not significantly restrict the size of any total homology group. The relevant diagonals of the double complex cross an unbounded, unidentified part of the spectral sequence. For further material on similar calculations with non-first-quadrant complexes, we direct the reader to \cite{bott1977cohomology}.
Furthermore, one may be interested to construct explicit cochains that generate the terms of the spectral sequence. To this end, it may be helpful to compare this to \cite[Section 3]{gelfand1972cohomology}; there, continuous cochains for~$\gl_n(\mathbb{K})$-Lie algebra bundles are constructed from elements in~$\left( \tfrac{\Omega^\bullet(M)}{ d \Omega^{\bullet -1}(M)} \right)^*$. When the bundle equals the endomorphism bundle~$\End(E)$ of some vector bundle~$E \to M$, these cochains equal the antisymmetrizations of the cyclic cochains constructed in \cite[Section 7]{quillen1988algebra}.
 
 Finally, we want to remark what would be necessary to remove the condition~$s \leq q$ in the description of~$E_{r,s}^2$ in Theorem~\ref{TheoremSpectralSequenceForGammaK}. There are essentially two problems
The first one is the lack of a full description of~$q$-diagonal, bornological Lie algebra homology of~$\gl(C^\infty_c(\mathbb{R}^n))$. We know that the~$q$-diagonal complex equals the usual one in degree~$r \leq q$, hence their homologies agree in degree~$r \leq q-1$. We have not yet explored higher degrees. 
Adapting our methods to the~$q$-diagonal case would force one to consider corresponding notions of~$q$-diagonal cyclic homology of~$C^\infty_c(\mathbb{R}^n)$, which could likely be calculated using the localization techniques for cyclic homology, see for example \cite{tele1998micro}. One would also require Künneth theorems for ``diagonal'' tensor products, which we have not defined here. Intuitively, a~$1$-diagonal tensor product would be of the following shape:
\begin{align*}
A \otimes_\Delta B := A \hotimes B / \, \overline{\langle a \otimes b : \supp a \cap \supp b = \emptyset \rangle}.
\end{align*}
It is not yet clear to us what such a Künneth theorem would look like.
The second problem is the lack of a full description of the unstable homology groups of~$\gl_n(C^\infty_c(\mathbb{R}^n))$. Already in the algebraic setting, this appears to be highly nontrivial:  Conjecture 10.3.9 in \cite{loday1998cyclic} attempts to give a description for Lie algebra homology of~$\gl_n(A)$ when~$A$ is commutative and unital, and it is stated that their conjecture implies a certain case of the Macdonald conjectures. In \cite{tele2003hodge}, the conjecture is verified for many special cases, but it is also shown that, in full generality, it does not hold. To our knowledge, a satisfying, general description of these unstable homology groups is an open problem.
\section{Proof of Theorem~\ref{TheoremNonUnitalLQT}}
\label{sec:ProofOfNonUnitalLQT}
The strategy and proof of Theorem~\ref{TheoremNonUnitalLQT} takes some preparation. The material and notation of this section originates from \cite{hanlon1988complete}, and we also refer to the related preprint \cite{cortinas2005cyclic}. There, the algebraic analogue of Theorem~\ref{TheoremNonUnitalLQT} was proven, and our contribution will be the extension of the results to nuclear Fréchet algebras. We closely follow the outline of the proof of  \cite{cortinas2005cyclic} and make remarks to how this extends to our setting at the appropriate places.
\begin{definition}
Let~$m,n \geq 1$.
\begin{itemize}
\item[i)]
 We say that~$\alpha := (\alpha_1,\dots,\alpha_l) \in \mathbb{Z}_{> 0}^l$ is a \emph{partition of~$m$ of length~$l = l(\alpha)$} if~$\sum_{i=1}^l \alpha_i = m$ and~$\alpha_1 \geq \dots \geq \alpha_l$. \\
Additionally, we define~$\emptyset$ to be a partition of 0 of length 0. \\
The set of partitions of~$m$ is denoted by~$P(m)$.
\item[ii)]
Let~$\alpha, \beta$ be two partitions of~$m$, and~$l(\alpha) + l(\beta) \leq n$, then we set
\begin{align*}
[\alpha, \beta]_n := (\alpha_1,\dots,\alpha_{l_1},0,\dots,0,-\beta_{l_2},\dots,-\beta_1) \in \mathbb{Z}^n.
\end{align*}
\item[iii)]
Let~$V$ be a~$\gl_n(\mathbb{K})$-representation and~$\mu \in \mathbb{Z}^n$, then we define the \emph{highest weight module}~$V_\mu$ via
\begin{align*}
M_\mu(V) &:= \{v \in V : e_{ij} \cdot v = 0, \, e_{kk} \cdot v = \mu_k \cdot v \quad \forall k, \, \forall i < j \}, \\
V_\mu &:= U(\gl_n(\mathbb{K})) \cdot M_\mu(V) .
\end{align*}
Here,~$e_{ij} \in \gl_n(\mathbb{K})$ denotes the elementary matrix with a one in the~$(i,j)$-th entry and zeroes everywhere else.
\end{itemize}
\end{definition}
\begin{lemma}
\label{LemmaDecompositionNonUnitalChains}
Let~$n, k \geq 1$, and~$A$ be any topological algebra.  Then:
\begin{align*}
C_k^\born(\gl_n(A))
= 
\bigoplus_{m \geq 0} 
\bigoplus_{ \substack{\alpha, \beta \in P(m) \\ l(\alpha) + l(\beta) \leq n }} 
C_k^\born(\gl_n(A))_{[\alpha,\beta]_n}.
\end{align*}
\end{lemma}
\begin{proof}
Consider~$\gl_n(\mathbb{K})$ and its tensor products~$\gl_n(\mathbb{K})^{\otimes^k}$ as a~$\gl_n(\mathbb{K})$-representation in the natural way, via the adjoint action and tensor products thereof. Since~$\gl_n(\mathbb{K})$ is a reductive Lie algebra, its adjoint representation is completely reducible, and the decomposition of the following finite-dimensional tensor modules is standard (for a detailled discussion, see \cite[p.211f.]{hanlon1988complete})
\begin{align*}
\gl_n(\mathbb{K})^{\otimes k} 
= 
\bigoplus_{m \geq 0} 
\bigoplus_{ \substack{\alpha, \beta \in P(m) \\ l(\alpha) + l(\beta) \leq n }} 
\left( \gl_n(\mathbb{K})^{\otimes k} \right)_{[\alpha,\beta]_n}.
\end{align*}
Now, since~$\gl_n(\mathbb{K})$ acts trivially on~$A$, this extends to a decomposition for the~$\gl_n(\mathbb{K})$-module~$\left( \gl_n(\mathbb{K})^{\otimes k} \otimes A^{\otimes^k} \right)_{\Sigma_k}$.
Lastly, since completion commutes with topological direct sums, this extends to~$C_k^\born(\gl_n(A))$ and the statement is shown.
\end{proof}
The action of a Lie algebra respects associated Chevalley-Eilenberg differentials, so we have the subcomplexes 
\begin{gather*}
C_\bullet^\born(\gl_n(A))_{[\alpha,\beta]_n} \subset C_\bullet^\born(\gl_n(A)), \\
 M_{[\alpha,\beta]_n} (C_\bullet^\born(\gl_n(A))) \subset M_{[\alpha,\beta]_n} (C_\bullet^\born(\gl_n(A))).
\end{gather*}
By reducing to the finite-dimensional highest weight theory as in Lemma~\ref{LemmaDecompositionNonUnitalChains}, one shows:
\begin{lemma}
\label{LemmaCalculationOfHighestWeightModules}
Let~$A$ be any topological algebra. For every~$n \geq 1, m \geq 0$ and~$\alpha, \beta \in P(m)$ with~$l(\alpha) +l (\beta) \leq n$, we have
\begin{align*}
H_\bullet(C^\born_\bullet(\gl_n(A))_{[\alpha,\beta]_n} ) \cong
U(\gl_n(\mathbb{K})) \cdot H_\bullet(M_{[\alpha,\beta]_n} (C^\born_\bullet(\gl_n(A)))).
\end{align*}
\end{lemma}
Some last definitions before we get to the relevant theorems: Given a partition~$\alpha \in P(m)$, we denote by~$V^\alpha$ the well-known \emph{Specht~$\Sigma_m$-module} associated to~$\alpha$, see \cite[p.216]{hanlon1988complete} for a detailled definition.
Lastly, given a chain complex~$C_\bullet$, denote by~$T^m(C_\bullet) := (C_\bullet)^{\otimes^m}$ the chain complex given by the~$m$-th algebraic tensor power of~$C_\bullet$.
The following will be our main proposition:
\begin{proposition}
\label{PropositionChainIsomorphismForNonunitalLQT}
Assume~$A$ is a nuclear Fréchet algebra. 
Then, for every~$n,m \geq 1$ and~$\alpha, \beta \in P(m)$ with~$l(\alpha) + l(\beta) \leq n$, there is a chain morphism
\begin{equation}
\label{MappingNonUnitalLQTIsomorphism1}
\begin{aligned}
&M_{[\alpha,\beta]_n} (C_\bullet^\born(\gl_n(A)))
\to
\\
&\quad \quad \quad \hat{\Lambda}^\bullet(C^{\lambda,\born}_{\bullet - 1}(A)) 
\otimes 
 \left(\hat{\tilde{T}}^m (C^{\barr, \born}_\bullet(A)) \otimes V^\alpha \otimes V^\beta \right)_{\Sigma_m},
 \end{aligned}
\end{equation}
which is an isomorphism of TVS in degree~$\leq \tfrac{n}{2}$. 
Here,~$\hat{T}$,~$\hat{\tilde{T}}$ and~$\hat{\Lambda}$ denote the completions of the (reduced/exterior) tensor algebra in the bornological tensor product topology, and the codomain is equipped with the product differential arising from the cyclic and bar differentials.
\end{proposition}
\begin{proof}
Assume first that~$A$ is an arbitrary TVS, not necessarily with an algebra structure. Set
\begin{align*}
S(A) := 
T^\bullet( \tilde{T}^{\bullet}(A)) \otimes (\tilde{T}^m(C_\bullet^\barr(A)) \otimes V^\alpha \otimes V^\beta).
\end{align*}
We first construct a map of graded vector spaces
\begin{align*}
\tilde{\psi}_A : 
S(A)
\to
M_{[\alpha,\beta]_n}
T^\bullet(\gl_n(A))
\end{align*}
as follows: 
Denote by~$e_{ij} \in \gl_n(\mathbb{K})$ the elementary matrix with a one in the~$(i,j)$-th entry and zeroes everywhere else. Set then for~$c = a_1 \otimes \dots \otimes a_p \in A^{\otimes^{p}}$ and~$1 \leq r,s \leq n$:
\begin{align*}
\zeta_{rs}(c) := 
\sum_{i_2,\dots,i_{p}} (e_{r i_2} \otimes a_1) \, \otimes  \, &(e_{i_2 i_3} \otimes a_2)
\otimes \dots \\ & \dots \otimes   (e_{i_{p-1} i_p} \otimes a_{p-1}) \otimes   (e_{i_{p} s} \otimes a_p) \in (\gl_n(A))^{\otimes^p}
\end{align*}
With this we define
\begin{align*}
\tilde{\theta} : T^\bullet(\tilde{T}^\bullet(A)) \to T^\bullet(\gl_n(A))
\end{align*}
as the graded algebra map induced by, for~$a_1 \otimes \dots \otimes a_p \in T^1(T^p(A))$:
\begin{align*}
\tilde{\theta} (a_1 \otimes \dots \otimes a_p) :&=
\sum_{1 \leq k \leq n}  \zeta_{kk}(a_1 \otimes \dots \otimes a_p)
\\
&=
\sum_{i_1,\dots,i_p} (e_{i_1 i_2} \otimes a_1) \otimes \dots \otimes (e_{i_n i_1} \otimes a_n).
\end{align*}
Further, recall that for some partition~$\gamma$ of a number~$m$, the Specht module~$V^\gamma$ is defined as generated by equivalence classes of standard Young tableaux of shape~$\gamma$ in a certain way \cite[p.216]{hanlon1988complete}. Two standard Young tableaux are considered equivalent if their rows contain the same numbers. 
Let~$x$ be such a Young diagram of length~$\leq n$ and~$1 \leq i \leq m$, then we set~$\rho_i(x)$ to be the row of~$x$ containing the number~$i$. We define:
\begin{align*}
\tilde{\epsilon} : \tilde{T}^m(C^\barr_\bullet(A)) \otimes V^{\alpha} \otimes V^\beta &\to T^\bullet(\gl_n(A)),\\
(c_1 \otimes \dots \otimes c_m) \otimes x \otimes y &\mapsto  \zeta_{\rho_1(x),n + 1 - \rho_1(y)} (c_1) \otimes \dots \otimes  \zeta_{\rho_m(x),n + 1 - \rho_m(y)} (c_m)
\end{align*}
Then, finally, we define 
\begin{align*}
\tilde{\psi}_A : S(A) \to M_{[\alpha,\beta]_n}
T^\bullet(\gl_n(A))
\end{align*} 
as the tensor product of~$\tilde{\theta}$ and~$\tilde{\epsilon}$. \cite[Theorem 3.4]{hanlon1988complete} shows that this indeed maps into the highest weight module~$M_{[\alpha,\beta]_n}
T^\bullet(\gl_n(A))$. 
Note also that~$\tilde{\psi}_A$ intertwines the actions
\begin{align}
\label{MappingActionsOnTheSpaceSA}
\mathbb{Z}/k_1 \mathbb{Z} \, \, \action \, \,  T^{k_1}(A), \quad 
\Sigma_{k_2} \, \, \action \, \,  T^{k_2}(\tilde{T}^{\bullet}(A)), \quad
\Sigma_{m} \, \, \action \, \,  (\tilde{T}^m(C_\bullet^\barr(A)) \otimes V^\alpha \otimes V^\beta)
\end{align}
on the domain with corresponding permutations of~$\Sigma_m$ on the codomain, and the invariants of both spaces with respect to these actions are
\begin{align*}
R^\Sigma(A) &:= \Lambda^\bullet(C^\lambda_{\bullet - 1}(A)) \otimes  \left( \tilde{T}^m (C^{\barr}_\bullet(A)) \otimes V^\alpha \otimes V^\beta \right)^{\Sigma_m},\\
R_\Sigma(A) &:= \Lambda^\bullet(C^\lambda_{\bullet - 1}(A)) \otimes  \left( \tilde{T}^m (C^{\barr}_\bullet(A)) \otimes V^\alpha \otimes V^\beta \right)_{\Sigma_m},
\end{align*}
and~$M_{[\alpha,\beta]_n} C_\bullet(\gl(A))$. Denote by~$\psi_A : R^\Sigma(A) \to C^\bullet(\gl(A))$ the arising map on invariants.
We show now:
\begin{lemma}
\label{LemmaAuxiliaryPsiAIsomorphism}
Let~$A$ be a complete TVS. Consider~$R^\Sigma(A)$ and~$M_{[\alpha,\beta]_n} (C_\bullet(\gl_n(A)))$ as TVS with the topology induced by the projective tensor product.
Then~$\psi_A$ extends to a morphism of topological graded vector spaces on the completions
\begin{align*}
\overline{R^\Sigma(A)}
 \to
M_{[\alpha,\beta]_n} (C_\bullet^\born(\gl_n(A))),
\end{align*}
which is an isomorphism in degree~$\leq \tfrac{n}{2}$. 
\end{lemma}
\begin{proof}
\cite[Theorem 3.6]{hanlon1988complete} shows that for every finite-dimensional~$A$, the map~$\psi_A$ is an isomorphism of vector spaces in degree~$\leq \tfrac{n}{2}$ .
Maschke's theorem implies that in characteristic zero and for a finite group~$G$, we can always assign to every equivariant morphism~$f : U \to V$ between~$G$-modules an equivariant map~$h : V \to U$, with the property that if~$f$ reduces to an isomorphism on invariants~$f^G: U^G \to V^G$, then~$h^G$ inverts~$f^G$. 
Hence in the case~$A = \mathbb{K}$ there is a map
\begin{align*}
\tilde{\kappa}_{\mathbb{K}} : 
T^\bullet(\gl_n(\mathbb{K}))
\to
S(\mathbb{K}),
\end{align*} 
intertwining the actions \eqref{MappingActionsOnTheSpaceSA},
whose reduction to invariants~$\kappa_\mathbb{K}: C_\bullet (\gl_n(\mathbb{K})) \to R(\mathbb{K})$ inverts~$\psi_\mathbb{K}$ in degree~$\leq \tfrac{n}{2}$. Note that the domain and codomain of~$\tilde{\kappa}_\mathbb{K}$ have finite-dimensional graded components and hence~$\tilde{\kappa}_\mathbb{K}$ and~$\kappa_\mathbb{K}$ are automatically continuous.
In the case of~$A$ being an arbitrary TVS, we can use canonical isomorphisms to identify
\begin{align*}
S(A) \cong S(\mathbb{K}) \otimes T^\bullet(A), \quad
T^\bullet(\gl_n(A)) \cong T^\bullet(\gl_n(\mathbb{K}))  \otimes T^\bullet(A).
\end{align*}
These isomorphisms identify~$\tilde{\psi}_A$ with~$\tilde{\psi}_{\mathbb{K}} \otimes \id_{T^\bullet(A)}$.
Under these identifications, consider~$\tilde{\kappa}_A := \tilde{\kappa}_{\mathbb{K}} \otimes \id_{T^\bullet(A)}$. The equivariance of~$\tilde{\kappa}_A$ and~$\tilde{\psi}_A$ under the actions \eqref{MappingActionsOnTheSpaceSA} show that the reduction to invariants~$\kappa_A: C_\bullet (\gl_n(A)) \to R(A)$ is a two-sided inverse to~$\psi_A$ in degree~$\leq \tfrac{n}{2}$, just as in the case~$A = \mathbb{K}$. This map is continuous as the tensor product of continuous maps.
Hence~$\psi_A$ has a continuous inverse, and thus lifts to a continuous morphism of graded TVS on the topological closures, and that this lift is an isomorphism in degree~$\leq \tfrac{n}{2}$. The lemma is shown.
\end{proof}
Now, assume~$A$ is a nuclear Fréchet algebra.
We use the previously constructed map~$\psi_A$, which only induced a morphism ofTVS, to induce a morphism of chain complexes. 
Denote by~$A^*$ the strongly continuous dual of~$A$ and by~$A^\vee$ the algebraic dual. Define the following maps for~$a \in A, \beta \in A^*, g, h \in \gl(\mathbb{K})$:
\begin{align*}
\text{ev} &: A \to (A^{\vee})^\vee,  \quad 
&&\text{ev}(a)(\beta) := \beta(a), \\
\nu &: \gl_n(A) \to \gl_n(A^\vee)^\vee, \quad
&&\nu(g \otimes a) (h \otimes \beta) := \tr (g^T \cdot h) \cdot \text{ev}(a)(\beta), \\
K &: C^{\lambda}_\bullet(A) \to C^{\lambda}_\bullet(A), \quad 
&&K([a_1 \otimes \dots \otimes a_n]) := n \cdot [a_1 \otimes \dots \otimes a_n].
\end{align*}
In \cite[Theorem 3.1]{cortinas2005cyclic}, a continuous graded chain map~$\phi$ is defined which makes the following diagram commute:
\[
\begin{tikzcd}
	M_{[\alpha,\beta]_n} 
	C_\bullet(\mathfrak{gl}_n(A)) 
	\arrow{d}{}
	\arrow{rrr}{\phi} 
	&&&
	R_\Sigma(A)
	\arrow{d}{}
	\\
	M_{[\alpha,\beta]_n} 
	C_\bullet(\mathfrak{gl}_n(A^\vee))^\vee 
	\arrow{rrr}{(\psi_{A^\vee} \circ (K \otimes 1))^T}
	&&&
	\left(R^\Sigma(A^\vee)\right)^\vee
\end{tikzcd}
\]
Here, the left vertical arrow is induced by~$\nu$, the right vertical arrow by~$\text{ev}$.
Since nuclear Fréchet algebras are automatically reflexive,~$\text{ev}$ and~$\nu$ are topological isomorphisms when all algebraic duals are replaced with strongly continuous duals and~$C_\bullet$ with~$C_\bullet^\born$.
We have also proven in Lemma~\ref{LemmaAuxiliaryPsiAIsomorphism} that~$\psi_{A^*}$ extends to an isomorphism of TVS in degree~$\leq \tfrac{n}{2}$ on the completion, and so does its transpose~$\psi_{A^*}^T$. Hence, if in the above diagram the algebraic duals are replaced with continuous duals and~$C_\bullet$ with~$C_\bullet^\born$, the bottom arrow, too, is a topological isomorphism. Hence all arrows except~$\phi$ in the above diagram induce topological isomorphisms in degree~$\leq \tfrac{n}{2}$ under the given replacements. Hence~$\phi$ does, too.
Hence the extension of~$\phi$ to the closures is a morphism of chain complexes and an isomorphism in degree~$\leq \tfrac{n}{2}$. This concludes the statement.
\end{proof}
\begin{proof}[Proof of Theorem~\ref{TheoremNonUnitalLQT}]
Due to the Lemmata~\ref{LemmaDecompositionNonUnitalChains} and~\ref{LemmaCalculationOfHighestWeightModules}, we have
\begin{align*}
H_\bullet^\born(\gl_n(A)) \cong 
H_\bullet &\big(C^\born_\bullet(\gl_n(A))_{[\emptyset,\emptyset]_n} \big)
\\
&\oplus
\bigoplus_{m \geq 1} 
\bigoplus_{ \substack{\alpha, \beta \in P(m) \\ l(\alpha) + l(\beta) \leq n }} 
U(\gl_n(\mathbb{K})) \cdot
H_\bullet \left( M_{[\alpha,\beta]_n}( C^\born_\bullet(\gl_n(A))) \right).
\end{align*}
If~$A$ is bornologically~$H$-unital, then its bornological bar complex is acyclic, and thus its differential has closed range. By assumption on~$A$, the differential of the cyclic complex has closed range, we can calculate the homology of the codomain of \eqref{MappingNonUnitalLQTIsomorphism1} via the Künneth isomorphism from Corollary~\ref{CorollarySchwartzKuenneth}. But acyclicity of the bar complex then implies acyclicity of this product complex. Since \eqref{MappingNonUnitalLQTIsomorphism1} is a chain isomorphism in degree~$r \leq \tfrac{n}{2}$, it induces isomorphisms of homology groups in degree~$2r  +1 \leq n$.
Thus, if~$2r + 1 \leq n$, we have 
\begin{align*}
H_r^\born(\gl_n(A)) \cong H_r \left(C^\born_\bullet(\gl_n(A))_{[\emptyset,\emptyset]_n} \right) = H_r \left(C^\born_\bullet(\gl_n(A))_{\gl_n(\mathbb{K})} \right),
\end{align*}
and the calculation of the homology of the invariant complex~$C^\born_\bullet(\gl_n(A))_{\gl_n(\mathbb{K})}$ is carried out exactly as in the unital case in Section~\ref{SectionLodayQuillenFrechet}. In particular, this homology stabilizes so that~$H_r^\iota(\gl(A)) \cong H_r^\iota(\gl_n(A)) \cong H_r^\born(\gl_n(A))$ when~$2r + 1 \leq n$.
Hence the statement is shown.
\end{proof}
\appendix
\section{The algebraic Loday-Quillen result}
\label{AppendixAlgebraicLodayQuillen}
We recall from the presentation in \cite[Chapter 10]{loday1998cyclic} the outline of the proof of the algebraic Loday-Quillen-Tsygan theorem, originally developed in \cite{loday1984cyclic} and even earlier in \cite{tsygan1983homology}.
Fix a unital algebra~$A$ in this section.
\begin{theorem}[Loday, Quillen, Tsygan]
Let~$A$ be a unital algebra and 
\begin{align*}
\gl(A) := \varinjlim \gl_n(A) := \varinjlim \gl_n(\mathbb{K}) \otimes A.
\end{align*}
 Then we have the following relation of the Lie algebra homology~$H_\bullet(\gl(A))$ and the cyclic homology~$H^\lambda_\bullet(A)$:
\begin{align*}
H_\bullet(\gl(A)) \cong \Lambda^\bullet H_{\bullet -1}^\lambda(A).
\end{align*}
All above tensor products and homologies are taken algebraically.
\end{theorem}
Due to the unitality of~$A$, for all finite~$n \in \mathbb{N}$ the Lie algebra~$\gl_n(A)$ contains the reductive subalgebra~$\gl_n(\mathbb{K})$, and thus the reduction~$C_\bullet(\gl_n(A)) \to C_\bullet(\gl(A))_{\gl_n(\mathbb{K})}$ is a quasi-isomorphism, see \cite[Prop 10.1.8]{loday1998cyclic}
\begin{proposition}
\label{PropositionInvariantTheoryOfGLn}
Denote by~$\Sigma_k$ the permutation group on~$k$ elements. \\
If~$n \geq k$, then there is an isomorphism 
\begin{align*}
\phi_n: \left(\gl_n(\mathbb{K}) ^{\bigotimes^k} \right)_{\gl_n(\mathbb{K})} \to \mathbb{K}[\Sigma_k],
\end{align*}
of~$\Sigma_k$-modules, where~$\Sigma_k$ acts on the left-hand side by permutation of tensor factors and on the right-hand side by the adjoint action.
\end{proposition}
We can make this isomorphism map explicit (see \cite[Chapter 9.2]{loday1998cyclic}): Define
\begin{align*}
g := g_1 \otimes \dots \otimes g_k \mapsto \sum_{\sigma \in \Sigma_k} T(\sigma)(g) \, \sigma \quad \forall g_i \in \gl_n(\mathbb{K}),
\end{align*}
where, if~$\sigma \in \Sigma_k$ assumes a cycle decomposition
\begin{align*}
\sigma = (i_1 \; \dots \; i_k) (j_1 \; \dots \; j_r) \dots (t_1 \; \dots \; t_s),
\end{align*}
we set
\begin{align*}
T(\sigma)(g) := \tr(g_{i_1} \dots g_{i_{k}}) \tr(g_{j_{1}} \dots g_{j_r}) \dots \tr(g_{t_{1}} \dots g_{t_{s}}).
\end{align*}
By the invariance of the trace under cyclic permutations, it is straightforward to show that this map factors through to~$C_\bullet(\gl_n(A))_{\gl_n(\mathbb{K})}$. A careful analysis then shows that this yields a bijection of the spaces.
The family of isomorphisms~$\{\phi_n\}_{n \geq k}$ is compatible with the inclusions induced by~$ \gl_n(\mathbb{K}) \to \gl_m(\mathbb{K})$ for~$m \geq n \geq k$, meaning the following diagram commutes:
\begin{figure}[h!]
\[
\begin{tikzcd}
	{\left(\mathfrak{gl}_n(\mathbb{K})^{\otimes^k} \right)_{\mathfrak{gl}_n(\mathbb{K})}} 
	\arrow{dr}{\phi_n}
	&& 
	{\left(\mathfrak{gl}_m(\mathbb{K})^{\otimes^k} \right)_{\mathfrak{gl}_m(\mathbb{K})}} 
	\arrow{dl}{\phi_m}
	\\
	& {\mathbb{K}[\Sigma_k]} &
	\arrow[from=1-1, to=1-3, hook]
\end{tikzcd}
\]
\end{figure}
Hence, since homology commutes with direct limits,  we get
\begin{align*}
H_\bullet(\gl(A)) 
\cong 
\varinjlim H_\bullet(\gl_n(A)) 
\cong
\varinjlim H_\bullet \left( C_\bullet(\gl_n(A))_{\gl_n(\mathbb{K})} \right),
\end{align*}
and as in the direct limit~$n \to \infty$, every graded component of~$C_\bullet(\gl_n(A))_{\gl_n(\mathbb{K})}$ becomes constant at some point, we can identify
\begin{align*}
\varinjlim  C_\bullet(\gl_n(A))_{\gl_n(\mathbb{K})} \cong C_\bullet(\gl(A))_{\gl(\mathbb{K})} 
\cong \bigoplus_{k \geq 0} \left( \mathbb{K}[\Sigma_k]  \otimes A^{ \otimes^k} \right)_{\Sigma_k}.
\end{align*}
The~$\Sigma_k$-action on the last term is given by the tensor product of the signed permutation action on~$A^{\otimes^k}$ and the adjoint action on~$\mathbb{K}[\Sigma_k]$
The last ingredient is to relate the last cochain complex space and the differential it inherits to the cyclic complex of~$A$ and the cyclic differential.
Recall that~$(1 \; \dotsb \; k) \in \Sigma_k$ denotes the cyclic permutation of~$k$ elements. 
\begin{proposition}
\label{PropositionThetaMapLodayQuillen}
Consider the map
\begin{align*}
C_{\bullet-1}^{\lambda}(A) &\to\bigoplus_{k \geq 0}  \left(  \mathbb{K}[\Sigma_{k}] \otimes A^{\otimes^k} \right)_{\Sigma_{k}},  \\
[a_1 \otimes \dotsb \otimes a_k]
&\mapsto 
[ (1 \; \dotsb \; k) \otimes (a_1 \otimes \dots \otimes a_{k}) ].
\end{align*}
This map is well-defined, intertwines the differentials, and extends to an isomorphism of chain complexes
\begin{align*}
\theta : 
\Lambda^\bullet  C_{\bullet-1}^{\lambda}(A) \to  \bigoplus_{k \geq 0} \left( \mathbb{K}[\Sigma_{k}] \otimes A^{\otimes^k} \right)_{\Sigma_{k}}
\end{align*}
by setting, for~$[u_1],\dots,[u_l]$  with~$[u_i] \in C_{k_i - 1}^\lambda(A)$ and~$N := \sum_i k_i$:
\begin{align*}
\theta([u_1] \wedge \dots \wedge [u_l])
:= 
\Big[\Big((&1 \; \dotsb \; k_1) \circ (k_1 + 1 \; \dotsb \; k_2) 
\circ 
\dots 
\circ (k_{l-1} + 1 \; \dotsb \; k_l) \Big) \\
&\otimes 
( u_1 \otimes \dots \otimes u_l) ] 
\in 
\left(\mathbb{K}\Big[\Sigma_{N} \Big] \otimes A^{\otimes^N} \right)_{\Sigma_{N}}.
\end{align*}
\end{proposition}
\begin{remark}
The reader is invited to check that the final product map is well-defined on all levels: It is independent of the ordering of~$[u_1] \wedge \dots \wedge [u_l]$, as a different ordering is equivalent to a permutation by~$\Sigma_N$, and we map into the~$\Sigma_N$-coinvariants. It is also independent of the choice of representatives~$u_i \in [u_i] \in C^\lambda_{k_i - 1}(A)$, since this is equivalent to a cyclic permutation acting on~$u_i$, and the cycle~$(k_{i-1} + 1 \; \dotsb \; k_i)$ is fixed under conjugation by itself.
\end{remark}
The differential on the domain of~$\theta$ is simply the differential of tensor product complexes. Hence, the Künneth theorem finally implies
\begin{align*}
\Lambda^\bullet H_{\bullet -1}^\lambda (A) \cong
\varinjlim H_\bullet \left( C_\bullet(\gl_n(A))_{\gl_n(\mathbb{K})} \right)
\cong
H_\bullet(\gl(A)).
\end{align*}
\section{The flat de Rham complex}
\label{AppendixFlatDeRhamPalamodov}
Let~$M$ be an smooth,~$n$-dimensional manifold (possibly with boundary), ~$N \subset M$ a closed submanifold, both of 
\begin{align*}
\dim H^k(N), \, \dim H^k(M) < \infty \quad \forall k \in \mathbb{N}_0.
\end{align*}
In this section, we want to make a short excursion in understanding the flat de Rham complex
\begin{align*}
\Omega_\fl^\bullet(M, N) := \{\omega \in \Omega^\bullet(M) : (j^\infty \omega)|_N = 0\},
\end{align*}
i.e. the subcomplex of~$\Omega^\bullet(M)$ given by forms which are flat on~$N$.
\begin{lemma}
\label{LemmaFlatDeRhamClosedAndComplemented}
Consider the complex
\begin{align}
\label{eq:RelativeFlatComplex}
0 \to \Omega^0_\fl(M,N) \to \dots \to \Omega^n_\fl(M,N) \to 0,
\end{align}
with the differential given by the restriction of the de Rham differential.
Its homology equals relative singular cohomology of the pair~$(M,N)$, and the differential has closed range.
\end{lemma}
The proof utilizes some basic knowledge of sheaves and sheaf cohomology, for which we direct the reader to \cite{bredon1997sheaf}.
\begin{proof}
The proof idea uses ideas from \cite{nicolaescu2016derham}.
Denote by~$\underline{\mathbb{R}}$ the constant sheaf on~$M$ and by~$\Omega^k$ the soft sheaf of~$k$-forms on~$M$. Then there is the well-known resolution
\begin{align*}
0 \to \underline{\mathbb{R}} \to \Omega^0 \to \Omega^1 \to \dots
\end{align*}
Assigning to a sheaf~$S$ on~$M$ the sheaf~$S_{M \setminus N}$ on~$M$ with stalks
\begin{align*}
(S_{M \setminus N})_x := 
\begin{cases}
S_x &\text{ if } x \not \in N, \\
0 &\text{ if } x \in N
\end{cases}
\end{align*}
is an exact functor, since exactness of sequences of sheaves may be checked on stalks, see \cite [Section I.2]{bredon1997sheaf}. Further, if~$S$ was soft, then so is~$S_{M \setminus N}$ \cite[Prop II.9.13]{bredon1997sheaf}.
Then the sheaf~$\Omega^\bullet_{M\setminus N}$ assigns to an open~$U \subset M$ the set~$\Omega^\bullet_\fl(U, U \cap N)$, and, if~$U \subset M$ is diffeomorphic to~$\mathbb{R}^n$ then for~$\underline{\mathbb{R}}_{M \setminus N}$ we have
\begin{align*}
\underline{\mathbb{R}}_{M \setminus N} (U) := 
\begin{cases}
\mathbb{R} & \text{ if } U \cap N = \emptyset, \\
0 & \text{ if } U \cap N \neq \emptyset.
\end{cases}
\end{align*}
Thus
\begin{align*}
0 \to \underline{\mathbb{R}}_{M \setminus N} \to \Omega^0_{M \setminus N} \to \Omega^1_{M \setminus N} \to \dots \to \Omega^n_{M \setminus N} \to 0
\end{align*}
is a soft resolution of the sheaf~$\underline{\mathbb{R}}_{M \setminus N}$, and the complex \eqref{eq:RelativeFlatComplex} calculates its sheaf cohomology.
But by standard sheaf-theoretical arguments using resolutions by singular co\-chain spaces, the sheaf cohomology of~$\underline{\mathbb{R}}_{M \setminus N}$ equals relative singular cohomology of the pair~$(M,n)$, see for example \cite[Chapter III.1]{bredon1997sheaf}.
Now, since we assumed~$M$ and~$N$ to have finite-dimensional cohomology, the image of the de Rham differential is, in every degree, cofinite-dimensional within in its kernel~$\ker d_\dR$. Together with Theorem~\ref{thm:SerreClosedRange} and closedness of~$\ker d_{k+1}$, this shows the statement.
\end{proof}
%
%
\begin{remark}
It is likely possible to drop the assumption of finite-dimensionality of the cohomology groups of~$M$ and~$N$ and proceed in a similar manner as \cite[Proposition 5.4]{palamodov1972stein}, but we do not need this here.
\end{remark}

\section{Cosheaves and \v{C}ech homology}
\label{AppendixCosheavesOnABase}
We require a few definitions regarding the sheaflike aspects of the space of bornological Lie algebra chains~$C_\bullet^\born(\Gamma_c(\mathcal{K}))$. Much of the material within this section can be read up on in more detail in \cite{bredon1997sheaf}. We assume the reader is familiar with the basic notions of sheaves on topological spaces. 
\begin{definition} \cite[Chapter V.1]{bredon1997sheaf}
Let~$M$ be a topological space.
\begin{itemize}
\item[i)]  A \emph{precosheaf} (of abelian groups)~$\mathcal{P}$ on~$M$ is a covariant functor from the category of open sets of~$M$, morphisms given by inclusions, into the category of abelian groups.\\
Given an inclusion~$U \subset V$ of open sets, we denote the associated mapping~$\mathcal{P}(U) \to \mathcal{P}(V)$ by~$\iota_U^V$, called the \emph{extension map} from~$U$ to~$V$ of the precosheaf~$\mathcal{P}$.
\item[ii)]
A \emph{cosheaf} is a precosheaf~$\mathcal{P}$ with the property that for every open cover~$\mathcal{U}$ of an open set~$U \subset M$, the sequence
\begin{align*}
\bigoplus_{i,j} \mathcal{P}(U_i \cap U_j) \to 
\bigoplus_i \mathcal{P}(U_i) \to
\mathcal{P}(U)
\to 0
\end{align*}
is exact, where the maps are given by
\begin{align*}
(a_{ij})_{i,j} \mapsto 
\left( \sum_j
\iota_{U_i \cap U_j}^{U_i} (a_{ij} - a_{ji})
\right)_i, \quad
(b_i)_i \mapsto \sum_i \iota_{U_i}^U b_i.
\end{align*}
We call a cosheaf~$\mathcal{P}$ \emph{flabby} if all extension maps~$\iota_U^V$ are injective.
\item[iii)] A \emph{morphism of (pre-)cosheaves} is a natural transformation between the functors defining the (pre-)cosheaves.
\end{itemize}
\end{definition}
Most practical examples of cosheaves arise from some notion of compactly supported objects, since compactly supported objects on some open set can always extended by zero to bigger open sets. The following proposition formalizes this:
\begin{proposition}
\label{PropositionSoftSheavesGiveFlabbyCosheaves}
\cite[ Proposition V.1.6]{bredon1997sheaf} Let~$\mathcal{S}$ be a sheaf a topological space~$M$, and consider the precosheaf~$\mathcal{S}_c$ which associates to an open~$U \subset M$ the set
\begin{align*}
\mathcal{S}_c(U) := \{ s \in \mathcal{S}(M) : \supp s \subset U \}
\end{align*}
and whose extension maps are given by extending by zero.
If~$\mathcal{S}$ is soft, then~$\mathcal{S}_c$ is a flabby cosheaf.
\end{proposition}
Completely dually to sheaf theory, one can define the \emph{\v{C}ech complex}~$\check{C}_\bullet(\mathcal{U};\mathcal{P})$ of a (pre-)cosheaf~$\mathcal{P}$ associated to an open cover~$\mathcal{U}$ of~$M$, which is then given as a chain complex
\begin{align*}
\dots \to \bigoplus_{i,j} \mathcal{P}(U_i \cap U_j) \to 
\bigoplus_i \mathcal{P}(U_i) \to 0,
\end{align*}
its differential being a skew-symmetric linear combination of the extension maps~$\iota_U^V$.
Its construction is fully dual to the standard, sheaf-theoretic \v{C}ech cochain complex, and we direct the reader to \cite[Chapter VI.4]{bredon1997sheaf} for details. We denote its homology by~$\check{H}_\bullet(\mathcal{U};\mathcal{P})$. The defining properties of a cosheaf~$\mathcal{P}$ directly imply~
\begin{align*}
\check{H}_0(\mathcal{U};\mathcal{P}) = \mathcal{P}(M)
\end{align*}
independently of the choice of~$\mathcal{U}$. 
Refinements of open covers induce on the associated \v{C}ech complexes the structure of a inverse system, and as such we may define the \v{C}ech homology of a cosheaf~$\mathcal{P}$ as the inverse limit of the \v{C}ech homologies of its associated open covers:
\begin{align*}
\check{H}_\bullet(M ; \mathcal{P}) := 
\varprojlim
\check{H}_\bullet(\mathcal{U} ; \mathcal{P}).
\end{align*}
Just like sheaf cohomology can be calculated in terms of resolutions, we shall calculate \v{C}ech homology in terms of \emph{coresolutions:}
\begin{definition}\cite[Chapter VI.7]{bredon1997sheaf}
\begin{itemize}
\item[i)] 
A precosheaf~$\mathcal{P}$ on~$M$ is called \emph{locally zero} if for every~$x \in M$ and every open neighbourhood~$U$ of~$x$ there is an open neighbourhood~$V \subset U$ so that~$\iota_U^V = 0$.
\item[ii)] A sequence of precosheaves
\begin{align*}
\mathcal{P}_1 \stackrel{f}{\to} \mathcal{P}_2 \stackrel{g}{\to} \mathcal{P}_3
\end{align*}
is called \emph{locally exact} ifthe precosheaf 
\begin{align*}
U \mapsto \im f(\mathcal{P}_1(U))/ \ker g (\mathcal{P}_2(U))
\end{align*}
is locally zero.
\item[iii)] A \emph{coresolution} of a cosheaf~$\mathcal{P}$ is a locally exact sequence of cosheaves
\begin{align*}
\dots \to \mathcal{P}_2 \to \mathcal{P}_1 \to \mathcal{P}_0 \to \mathcal{P} \to 0.
\end{align*}
The coresolution is called \emph{flabby} if the~$\mathcal{P}_0,\mathcal{P}_1,\dots$ (but not necessarily~$\mathcal{P}$) are flabby.
\end{itemize}
\end{definition}
To calculate \v{C}ech homology of cosheaves, the following result will be helpful:
\begin{proposition} \cite[Thms VI.7.2, VI.13.1]{bredon1997sheaf}
\label{PropositionConnectionCechAndCoresolutionHomology}
Let~$\mathcal{P}$ be a cosheaf on~$M$ with flabby coresolution
\begin{align*}
\dots \to \mathcal{P}_2 \to \mathcal{P}_1 \to \mathcal{P}_0 \to \mathcal{P} \to 0.
\end{align*}
\begin{itemize}
\item[i)] The \v{C}ech homology~$\check{H}_\bullet(M;\mathcal{P})$ is equal to the homology of the complex
\begin{align*}
\dots \to \mathcal{P}_2(M) \to \mathcal{P}_1(M) \to \mathcal{P}_0(M) \to 0.
\end{align*}
\item[ii)] 
If~$\mathcal{U}$ is an open cover of~$M$ with the property that
\begin{align*}
\dots \to \mathcal{P}_2(U) \to \mathcal{P}_1(U) \to \mathcal{P}_0(U) \to \mathcal{P}(U) \to 0
\end{align*}
is exact whenever~$U$ is a finite intersection of elements of~$\mathcal{U}$, then 
\begin{align*}
\check{H}_\bullet(\mathcal{U};\mathcal{P}) = \check{H}_\bullet(M;\mathcal{P}).
\end{align*}
\end{itemize}
\end{proposition}
\begin{corollary}
\label{CorollaryFlabbyCosheavesTrivialCech}
For every flabby cosheaf~$\mathcal{P}$ on~$M$, and every open cover~$\mathcal{U}$ of~$M$, we have
\begin{align*}
H_r(M,\mathcal{P}) = H_r(\mathcal{U},\mathcal{P}) = 
\begin{cases}
\mathcal{P}(M) &\text{ if } r = 0, \\
0 &\text{ else.}
\end{cases}
\end{align*}
\end{corollary}
\begin{proof}
Consider the flabby coresolution~$0 \to \mathcal{P} \stackrel{\id}{\to} \mathcal{P} \to 0$ and apply Proposition~\ref{PropositionConnectionCechAndCoresolutionHomology}.
\end{proof}
One concept which one might hope for in the theory of cosheaves is a dual version of the well-known concept of sheafification, in other words, a way to universally assign to every precosheaf an appropriate cosheaf. For sheaves, one speaks of a left-adjoint functor to the inclusion of presheaves into sheaves, and sheafification exists for presheaves in most standard categories, e.g. the category of sets or abelian groups. Since sheafification respects stalks, locally, the original presheaf and its associated sheafification carry the same information.
Surprisingly, the dual concept of ``cosheafification'' is a lot more involved, and even existence of this concept in most standard categories is a difficult question, let alone constructing it explicitly, see for example \cite{curry2014sheaves}.
Instead, we will consider the concept of a \emph{cosheaf on a base}. While the dual notion of sheaves on a base is well-studied, we are not aware of any mention in the literature of the cosheaf-theoretic version thereof.
\begin{definition}
Let~$\mathcal{B}$ be a topological base of~$M$. In the following, view~$\mathcal{B}$ as a subcategory of the category of open sets of~$M$. 
\begin{itemize}
\item[i)] A \emph{precosheaf~$S$ on }~$\mathcal{B}$ is a covariant functor from~$\mathcal{B}$ to the category of abelian groups. We denote the image of~$U \in \mathcal{B}$ as~$S(U)$ and the arising extension maps for~$U \subset V \in \mathcal{B}$ by~$\iota_U^V$.
\item[ii)] Choose for any~$U \in \mathcal{B}$ an open cover~$\{U_i\}_{i \in I}$ by elements in~$\mathcal{B}$, and for every~$i,j \in I$ an open cover ~$\{V_{ij,k}\}_{k \in K}$ of~$U_i \cap U_j$ by elements in~$\mathcal{B}$. We call a precosheaf~$S$ on~$\mathcal{B}$ a \emph{cosheaf on~$\mathcal{B}$} if, for all such choices, the following sequence is exact:
\begin{align*}
0 \leftarrow
P(U)
\leftarrow
\bigoplus_i P(U_i)
\leftarrow
\bigoplus_{ijk} P(V_{ij,k}).
\end{align*}
\item[iii)] A morphism of (pre-)cosheaves on~$\mathcal{B}$ is a natural transformation of the functors defining the (pre-)cosheaves.
\end{itemize}
\end{definition}
The sequence is the analogue of the cosheaf condition, but rather than working with all open sets, we just work with elements of a topological base~$\mathcal{B}$. If~$\mathcal{B}$ is chosen as the topology of~$M$, then this definition is equivalent to the definition of a cosheaf on~$M$. 
This is precisely the dual of the well-studied concept of sheaves on a base, by viewing~$\mathrm{Ab}$-valued cosheaves as~$\text{Ab}^{\text{op}}$-valued sheaves. 
\begin{theorem}
\label{thm:CosheavesOnBaseExtendToCosheavesOnSpace}
Given a topological space~$M$ and a topological base~$\mathcal{B}$ of~$M$. An~$\mathrm{Ab}$-valued cosheaf on~$\mathcal{B}$ extends, up to cosheaf isomorphism, uniquely to a cosheaf on~$M$. A morphism between two cosheaves on~$\mathcal{B}$ of~$M$ extends uniquely to a morphism between the induced cosheaves on~$M$.
\end{theorem}
\begin{proof}
The following proof is due to \cite{jgon2020cosheaf}.
The analogue statement for~$\mathcal{C}$-valued sheaves is true whenever~$\mathcal{C}$ is a complete category (see \cite{vakil2017rising} or \cite[Lemma 2.2.7]{qing2002algebraic}). 
However, since Ab is a cocomplete category,~$\text{Ab}^{\text{op}}$ is a complete category. This proves the statement.
\end{proof}
It is known that the setwise cokernels of cosheaf morphisms are again cokernels \cite[Prop VI.1.2]{bredon1997sheaf}, the proof being a simple diagram chase. This straightforwardly extends to cosheaves on a base:
\begin{proposition}
\label{prop:CokernelCosheaves}
Let~$\mathcal{B}$ be a topological base of~$M$.
Let further~$\phi : \mathcal{P} \to \mathcal{S}$ be a morphism of cosheaves on~$\mathcal{B}$, and define a precosheaf~$\coker \phi$ by assigning to~$B \in \mathcal{B}$
\begin{align*}
\coker \phi (B) := \mathcal{S}(B) / \phi(\mathcal{P}(B)),
\end{align*} 
with extension maps induced by the cosheaf maps of~$\mathcal{S}$. Then~$\coker \phi$ defines a cosheaf on~$\mathcal{B}$.
\end{proposition}
\bibliographystyle{unsrt}
\bibliography{lit}
\end{document}